\documentclass[a4paper,12pt,reqno]{amsart}
\usepackage{amsmath,amsfonts,amssymb,amsthm,enumerate,multicol}
\usepackage[pagewise]{lineno}
\usepackage{tikz,graphicx}
\usepackage{hyperref}
\usepackage{caption,enumitem,wasysym}
\usepackage{cleveref}
\usepackage{epstopdf}
\usepackage{float,wrapfig}
\usepackage[labelsep=space]{caption}
\usepackage[font=footnotesize]{caption}
\usepackage{xcolor}
\usepackage{cite}
\usepackage[autostyle]{csquotes}

\newtheorem{theorem}{Theorem}[section]
\newtheorem{corollary}[theorem]{Corollary}
\newtheorem{lemma}[theorem]{Lemma}

\newtheorem{definition}[theorem]{Definition}
\newtheorem{remark}[theorem]{Remark}

\numberwithin{equation}{section}

\allowdisplaybreaks

\theoremstyle{plain}

\usepackage{mathtools}

\usepackage{amsmath,amsfonts,amssymb,amsthm,enumerate,multicol}
\usepackage{tikz}
\usepackage{float}
\allowdisplaybreaks

\numberwithin{equation}{section}
\begin{document}
	\begin{center}
		\small{\textbf{Existence and Non-existence for Exchange-Driven Growth Model}}
	\end{center}

	

	

	\medskip
	\medskip
	\centerline{${\text{Saroj ~ Si$^{\dagger}$}}$ and ${\text{Ankik ~ Kumar ~Giri$^{\dagger,\ast}$}}$}\let\thefootnote\relax\footnotetext{$^{\ast}$Corresponding author. Tel +91-1332-284818 (O);  Fax: +91-1332-273560  \newline{\it{${}$ \hspace{.3cm} Email address: }}ankik.giri@ma.iitr.ac.in}
	\medskip
	{\footnotesize

		\centerline{ ${}^{\dagger}$ Department of Mathematics, Indian Institute of Technology Roorkee,}
		\centerline{Roorkee-247667, Uttarakhand, India}
	}

	\bigskip

	\begin{quote}
		{\small {\em \bf Abstract.}  The exchange-driven growth (EDG) model describes the evolution of clusters through the exchange of single monomers between pairs of interacting clusters. The dynamics of this process are primarily influenced by the interaction kernel $K_{j,k}$. In this paper, the global existence of classical solutions to the EDG equations is established for non-negative, symmetric interaction kernels satisfying $K_{j,k} \leq C(j^{\mu}k^{\nu} + j^{\nu}k^{\mu}) $, where $\mu, \nu \leq 2$, $\mu + \nu \leq 3$, and $C>0$, with a broader class of initial data. This result extends the previous existence results obtained by Esenturk \cite{Esenturk}, Schlichting \cite{Schlichting}, and Eichenberg \& Schlichting \cite{Eichenberg}.
		Furthermore, the local existence of classical solutions to the EDG equations is demonstrated for symmetric interaction kernels that satisfy $K_{j,k} \leq C j^{2} k^{2}$ with $C > 0$, considering a broader class of initial data. In the intermediate regime $3 < \mu + \nu \leq 4$, the occurrence of finite-time gelation is established for symmetric interaction kernels satisfying $C_{1}\left(j^{2}k^{\alpha}+j^{\alpha}k^{2}\right)\leq K_{j,k}\leq Cj^{2}k^{2}$, where $1 < \alpha \leq 2$, $C>0$, and $C_{1} > 0$, as conjectured in \cite{Esenturk}. In this case, the non-existence of the global solutions is ensured by the occurrence of finite-time gelation. Finally, the occurrence of instantaneous gelation of the solutions to EDG equations for symmetric  interaction kernels satisfying $K_{j,k}\geq C\left(j^{\beta}+k^{\beta}\right)$ ($\beta>2, C>0)$ is shown, which also implies the non-existence of solutions in this case.}
	\end{quote}

	\vspace{.3cm}
	
	

	\section{Introduction}

	\noindent Exchange-driven growth (EDG) processes arise in the dynamics of cluster growth through the exchange of single units called monomers. It is usually assumed that clusters are completely identified by their mass or volume. The EDG model has applications in migration \cite{Ke1}, wealth exchange \cite{Isp}, population dynamics \cite{Leyvraz} and mean-field limit of a class of interacting particle systems \cite{Stef2, Lam}; see also \cite{Esenturk, Schlichting, barik2024discrete, Esenturk-L} and the references therein. If $f_{j}(t)$ denotes the density of clusters of size $j$ at time $t$, then the basic reaction for the exchange-driven growth phenomenon can be interpreted as
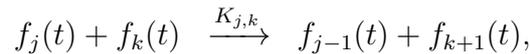
\begin{align*}
		f_{j} (t)+ f_{k}(t)\;\;\xrightarrow{\scriptsize{K_{j,k}}}{}\;\;  f_{j-1}(t)+ f_{k+1}(t),
	\end{align*}	
for $t\ge 0$ and $j,k\in \mathbb{N}_{0}:= \mathbb{N}\cup\left\{0\right\}$ with $\mathbb{N}:=\left\{1, 2, 3, \ldots\right\}$. Here, the interaction kernel $K_{j,k}$ depends only on the size of the clusters involved in the reaction. More precisely, in the phenomenon described above, the clusters of size $j$ export a monomer to the clusters of size $k$. Mathematically, the exchange-driven growth phenomenon is represented by an infinite system of non-linear ordinary differential equations
	\begin{align}
		\text{ }\dot{f}_{j}(t)  &  = f_{j+1}(t)\sum_{k=0}^{\infty}K_{j+1,k} f_{k}(t)- f_{j}(t) \sum_{k=0}^{\infty}K_{j,k} f_{k}(t)\label{infode}\\
		& \quad - f_{j}(t)\sum_{k=1}^{\infty}K_{k,j} f_{k}(t)+ f_{j-1}(t)\sum_{k=1}^{\infty
		}K_{k,j-1} f_{k}(t),\quad \text{for}\quad j\in \mathbb{N},\nonumber
		\end{align}
	and
	\begin{align}
		\quad \dot{ f}_{0}(t)&= f_{1}(t)\sum_{k=0}^{\infty}K_{1,k} f_{k}(t)- f_{0}(t)\sum_{k=1}^{\infty
		}K_{k,0} f_{k}(t)\text{,} \label{0-infode}
	\end{align}
with initial data
\begin{align}
	 f_{j}(0)&= f_{j,0}\quad \text{for}\quad j\in \mathbb{N}_{0}.\label{infIC}
\end{align}
The first and fourth terms in \eqref{infode} represent the formation of clusters of size $j$ when either clusters of size $j+1$ release a monomer to other clusters, or clusters of size $j-1$ acquire a monomer from other clusters. The second and third terms in \eqref{infode} represent the loss of clusters of size $j$ when the clusters of size $j$ either export a monomer to, or import a monomer from other clusters. In the same vein, the first term in \eqref{0-infode} represents the birth of clusters of size $0$, while the second term represents their death. There is no loss of mass or number of particles in the exchange-driven growth phenomenon. Therefore, the total mass and the total number of particles of the EDG system \eqref{infode}--\eqref{infIC} are expected to be conserved, i.e.,
\begin{align*}
	\sum_{j=1}^{\infty} jf_{j}(t)=\sum_{j=1}^{\infty} jf_{j}(0) \quad \text{ and}\quad \sum_{j=0}^{\infty} f_{j}(t)=\sum_{j=0}^{\infty} f_{j}(0),
\end{align*}
 respectively, for $t\ge 0$. However, Ben-Naim \&  Krapivsky \cite{Naim} observe that when the growth of the interaction kernel $K_{j,k}$ is high, the conservation of total mass may fail due to the formation of an infinite-size mass, or gel. This phenomenon is known as \emph{gelation} \cite{Ball, Ziff, Esco, Menon}. Mathematically, the \emph{gelation time} $T_{gel}$ is defined by
\begin{align*}
	T_{gel}:= \inf \left\{ t\ge 0 : \sum_{j=1}^{\infty} jf_{j}(t) <\sum_{j=1}^{\infty} jf_{j}(0)\right\}.
\end{align*}
	It is well-known that if $\sum_{j=1}^{\infty} j^{m}f_{j}(t_{0})=\infty$ for some $m\in \mathbb{N}$ and $t_{0} \in (0,\infty)$, then $T_{gel}\leq t_{0}$, see \cite[Lemma 9.2.2]{BLL_book}. In other words, $\sum_{j=1}^{\infty} j^{m}f_{j}(t)<\infty$ for all $m\in \mathbb{N}$ whenever $t<T_{gel}$. In the context of exchange-driven growth, a preliminary investigation in \cite{Naim} suggests that symmetric interaction kernels of the form  $K_{j,k}=(jk)^{\eta}$ do not lead to gelation if $\eta \leq 3/2$. However, when $ \frac{3}{2} < \eta \leq 2$, gelation occurs at a finite time while gelation occurs immediately at $t = 0$ for $\eta > 2$, which is known as \emph{instantaneous gelation}. This occurrence of gelation and instantaneous gelation in exchange-driven growth model differ from the Smoluchowski model. In Smoluchowski's coagulation model \cite{ernst1984}, gelation and instantaneous gelation occur for $1/2 < \eta \leq 1$ and $\eta >1$, respectively, see \cite{Dongen, Carr}. 
	
	The mathematical study of the infinite EDG system \eqref{infode}--\eqref{infIC} was initiated by Esenturk \cite{Esenturk} with the well-posedness result for the interaction kernels growing at most linearly; that is, 
	\begin{align}\label{kerlinear}
	0\leq K_{j,k} \leq C jk\quad \text{for all}\; j, k \in \mathbb{N}_{0}.
	\end{align}
In this case, the global existence of solution is shown by taking an additional assumption on the initial data, i.e., $\sum_{j=1}^{\infty} j^{p}f_{j}(0)<\infty$, for some $p>1$ and the uniqueness is shown with $\sum_{j=1}^{\infty} j^{2}f_{j}(0)<\infty$. In addition, for symmetric interaction kernels, the global well-posedness has been extended in \cite{Esenturk} to the interaction kernels satisfying 
	\begin{align}\label{quadraticcoagassum}
	0\leq K_{j,k}=K_{k,j} \leq C_{q} \left( j^\mu k^\nu + j^\nu k^\mu\right)\; \text{ for}\; j, k \in \mathbb{N}_{0}\;\text{ with}\; \mu,\nu \in [0,2]\; \text{and} \; \mu+\nu\leq 3.
	\end{align}
In this case, the global existence of solution is shown by taking an additional assumption on the initial data, i.e., $\sum_{j=1}^{\infty} j^{p}f_{j}(0)<\infty$, for some $p>2$ and the uniqueness is shown with $\sum_{j=1}^{\infty} j^{4}f_{j}(0)<\infty$. Moreover, a non-existence result has been established in \cite{Esenturk} for the faster growing symmetric interaction kernels, i.e.,
\begin{align}\label{superkernel} 
0\leq K_{j,k}=K_{k,j}\ge Cj^{\beta} \quad \text{with}\quad K_{j,0}=0,\quad j\in \mathbb{N}_{0},
\end{align}
where $C>0$ and $\beta >2$ with some additional assumption on the initial data, i.e., $\lim_{n\rightarrow\infty}e^{\delta n^{\beta -2}}\sum_{j=n}^{\infty}(j^2-n^2) f_{j}(0)\nrightarrow 0$. Later, Schlichting \cite{Schlichting} extended the global existence result for the class of interaction kernels satisfying \eqref{kerlinear} and
\begin{align*}
K_{j,k-1}\leq C_{K} jk\; \quad \text{for}\; j,k \in \mathbb{N},
\end{align*}
 by relaxing the additional assumption taken on the initial data, specifically, the requirement that $\sum_{j=1}^{\infty} j^{p}f_{j}(0)<\infty$, for some $p>1$. Eichenberg \& Schlichting \cite{Eichenberg} established the global and local existence of solutions for a particular class of homogeneous, product-type separable interaction kernels of the following form
\begin{align}\label{homogeneouskernel}
	K_{j,k}= \begin{cases}
		(jk)^{\eta}, \quad \text{when}\; \eta >0, \\
		1-\delta_{j,0}, \quad \text{when}\; \eta =0,
	\end{cases}
\end{align}
for all $j, k \in \mathbb{N}_{0}$ and $\eta \in [0,2]$. The global existence of solutions for $0\leq \eta \leq 3/2$ and the local existence of solutions for $3/2< \eta \leq 2$ to \eqref{infode}--\eqref{infIC} have been demonstrated for a generalized class of initial data in the work of Eichenberg and Schlichting. However, the non-homogeneous interaction kernels and the homogeneous interaction kernels of the form $K_{j,k}=C\left( j^\mu k^\nu + j^\nu k^\mu\right)$ with $\mu \neq \nu$ that satisfy \eqref{quadraticcoagassum} do not fall in the class of interaction kernels that satisfy \eqref{homogeneouskernel}. To the best of our knowledge, the existence result in \cite{Esenturk} for this general class of interaction kernels has not yet been improved for a larger class of initial data.

The main aim of the present work is to extend the existence results established in \cite{Esenturk, Eichenberg, Schlichting} and examine the occurrence of gelation and the phenomenon of instantaneous gelation. Specifically, the existence result provided in \cite{Esenturk} for the class of symmetric interaction kernels \eqref{quadraticcoagassum} is generalized by relaxing the more restrictive condition imposed on the initial data, specifically, the requirement that $ M_{p}(f(0)) < \infty$ for some  $p > 2$. To prove the global and local existence of solutions, the same approach as in \cite{Esenturk,McLeod,McLeod2} is followed, namely, the truncation method, and some moment estimates are derived. Then, to find the bounds on higher moments, the de la Vallée-Poussin theorem is applied, and some convex function inequalities derived in \cite{laurenccot2014weak} are utilized. This helps to remove the additional assumption taken on the initial data in \cite[Theorem 2]{Esenturk}, thereby extending the global and local existence results of \cite{Eichenberg, Schlichting}. Moreover, an intermediate regime ($\mu, \nu \leq 2$, $\mu + \nu \leq 4$) for $K_{j, k}$ is identified, where the solutions exist locally for a more general class of initial data. The conjecture made in \cite{Esenturk} is proved to some extent, indicating the gelation regime in which blow-up of the certain moment of the solution occurs after a finite time (the gelation time). The occurrence of finite time gelation is shown with the help of the propagation of moments result \Cref{momentpropagation} and Gronwall's inequality. The non-existence of global solutions to \eqref{infode}--\eqref{infIC} is concluded based on the occurrence of finite time gelation for the symmetric interaction kernel satisfying $C_{1}\left(j^{2}k^{\alpha}+j^{\alpha}k^{2}\right) \leq K_{j,k} \leq Cj^{2}k^{2}$ , where $1 < \alpha \leq 2$ and  $C_{1} > 0$. It is also demonstrated that, for the class of interaction kernels $K_{j,k} = C j^{2} k^{2}$ with $M_{2}(f(0)) > 0$ and $C > 0 $, the gelation time is given by $T_{\text{gel}}=\left(2M_{2}(f(0))C\right)^{-1}$. Beyond this regime, i.e., if  $ K_{j,k} \geq Cj^{\beta}$ ( $\beta > 2$ and $C>0$), the occurrence of instantaneous gelation is shown, i.e., $T_{\text{gel}} = 0$. The occurrence of the instantaneous gelation is based on the idea of obtaining lower bounds for the higher moments of the solution. The occurrence of instantaneous gelation plays a key role in proving the non-existence result established in \cite{Esenturk} for the class of symmetric interaction kernels satisfying \eqref{superkernel}.
	
The structure of this article is as follows: In Section 2, we provide the notion of the solution to \eqref{infode}--\eqref{infIC} and state the main results of this article. Section 3 is devoted to proving the global existence of a solution for the symmetric interaction kernels, as mentioned in \Cref{mainthm-1}, and to establishing the local existence of a solution, as described in \Cref{localsolution}. The occurrence of the gelation phenomenon, which prevents the existence of global mass-conserving solutions, is addressed in Section 4. In Section 5, the occurrence of instantaneous gelation is established, which guarantees the non-existence of solutions to \eqref{infode}--\eqref{infIC} for any time.
\section{Notion of Solution and Main Results}
	\noindent We first introduce some notations. Define $Y_{r}^{+}$ as the positive cone of the Banach space  $\left(Y_{r}, \| \cdot \|_{r}\right)$, where
	\begin{align*}
	Y_{r} := \left\{ f = (f_{j})_{j \ge 0} : \|f\|_{r} = \sum_{j=0}^{\infty} j^{r} |f_{j}| < \infty \right\},
    \end{align*}
	for  $r \geq 0$. In other words, $ Y_{r}^{+} := \left\{ f \in Y_{r} : f_{j} \geq 0 \text{ for all } j \in \mathbb{N}_{0} \right\}$, where $\mathbb{N}_{0} := \{0, 1, 2, \dots\}$.
	\begin{definition}\label{notionofsolution}
		Let $T\in (0,\infty]$ and $f(0):=(f_{j}(0))_{j\ge 0}\in Y_{r}^{+}$. A mild solution $f=(f_{j})_{j\ge 0}$ to the EDG system \eqref{infode}--\eqref{infIC} on $[0,T)$, with the initial data $f(0)=(f_{j}(0))_{j\ge 0}$, is a function $f: [0,T)\mapsto Y_{r}^{+}$ such that
		\begin{enumerate}
			\item For all $j\in\mathbb{N}_{0}$, $ f_{j}:[0,T)\rightarrow\lbrack0,\infty)$ is continuous and
			$\sup_{t\in [0,T)} \|f(t)\|_{r}<\infty$,
			\item For all $j\in\mathbb{N}_{0}$, $t\in\lbrack0,T)$,  $\int_{0}^{t}\sum_{k=0}^{\infty}K_{j,k} f_{k}ds<\infty,$
			\item For all $j\in\mathbb{N}$, $t\in\lbrack0,T)$, the following hold 
			\begin{align*}
				f_{j}(t)&= f_{j}(0)+\int_{0}^{t}   f_{j+1}(s)\sum_{k=0}^{\infty
				}K_{j+1,k} f_{k}(s) ds- \int_{0}^{t} f_{j}(s)\sum_{k=0}^{\infty}K_{j,k} f_{k}(s)  ds\\
				&\quad -\int_{0}^{t} f_{j}(s)\sum_{k=1}^{\infty}K_{k,j} f_{k}(s) ds+ \int_{0}^{t} f_{j-1}(s)\sum_{k=1}^{\infty}K_{k,j-1} f_{k}(s) ds, \nonumber\\ 
		 f_{0}(t)&= f_{0}(0)+\int_{0}^{t} f_{1}(s)\sum_{k=0}^{\infty}
				K_{1,k} f_{k}(s) ds- \int_{0}^{t} f_{0}(s)\sum_{k=1}^{\infty}K_{k,0} f_{k}(s)ds.
			\end{align*}
		\end{enumerate}
	\end{definition}
	Let us now define the $p$-th moment of the solution $\left(f_j\right)_{j\ge 0}$ as
	\begin{align*}
	M_{p}(f):=\sum_{j=0}^{\infty}j^{p} f_{j}\quad \text{and}\quad M_{p}^{N}(f):=\sum_{j=0}^{N}j^{p} f_{j},
	\end{align*}
	where $p >0$, $N\in \mathbb{N}$ and the zeroth moment as $M_{0}(f):=\sum_{j=0}^{\infty} f_{j}$ and $M_{0}^{N}(f):=\sum_{j=0}^{N} f_{j}$.
	In particular, at time $t$, $M_{0}(f(t))$ and $M_{1}(f(t))$ represent the total number of particles and the total mass of the system, respectively.
	
	We are now ready to outline the primary outcomes of this paper.
	\begin{theorem}[Existence]\label{mainthm-1}
 Let the interaction kernel $K_{j,k}$ satisfy \eqref{quadraticcoagassum}, and let $f(0)\in Y_{1}^{+}$ with the finite initial mass $\rho:= M_{1}(f(0))<\infty$.
		\begin{enumerate}[label=(\alph*)]
			\item If  $f(0)=(f_{j}(0))_{j\ge 0} \in Y_{\lambda}^{+}$, then
			the system (\ref{infode})-(\ref{infIC}) has a global mild solution $f=(f_{j})_{j\ge 0}$ such that  $f(t)\in Y_{\lambda}^{+}$ for each $t\in [0, \infty)$ with
			\begin{align}\label{LambdaD}
				\lambda:= \max\{\mu,\nu\}>1.
			\end{align}
			\item If $\max\{\mu,\nu\} \leq 1$, then the system \eqref{infode}--\eqref{infIC} has a global mild solution $f=(f_{j})_{j\ge 0}$ such that  $f(t)\in
			Y_{1}^{+}$ for each $t\in [0, \infty)$.
		\end{enumerate}
	\end{theorem}
	\begin{remark}
		The main contribution of \Cref{mainthm-1} is the selection of the initial data space or the solution space according to the precise growth of the interaction kernel $K_{j,k}$. This approach eliminates the need for the more restrictive condition on the initial data, specifically, $M_{p}(f(0)) < \infty $ for some  $p > 2 $, as required in \cite[Theorem 2]{Esenturk}.	
	\end{remark}
	The proof of \Cref{mainthm-1} is based on the construction of approximating solutions by defining a truncated system of equations \eqref{Tode0}--\eqref{TodeIC} for the infinite EDG system \eqref{infode}--\eqref{infIC}. Next, we establish uniform bounds for both the approximating solutions and their derivatives. With the help of these bounds and the Arzelà-Ascoli theorem, a convergent subsequence is extracted from the sequence of approximating solutions. Furthermore, it is shown that the limit function of this convergent subsequence is indeed a mild solution to \eqref{infode}--\eqref{infIC}. The main challenge in the existence proof is to pass the limit in the truncated system \eqref{Tode0}--\eqref{TodeIC}. To achieve this, we need bounds for higher moments, which are obtained through the implementation of a refined version of the de la Vall\'{e}e-Poussin theorem initially presented in \cite[Proposition~I.1.1]{Le1977} and later used in \cite[Theorem~7.1.6]{BLL_book}. Moreover, by improving the regularity of the mild solution, we establish the global existence of the classical solution that conserves both the total mass and the total number of particles in the system. 
	\begin{corollary}\label{regularityandconservation}
		Let $f=(f_{j})_{j\ge 0}$ be the mild solution to \eqref{infode}--\eqref{infIC} under the
		conditions of \Cref{mainthm-1}. Then $f=(f_{j})_{j\ge 0}$ is a continuously
		differentiable solution to \eqref{infode}--\eqref{infIC} satisfying  $M_{0}(f(t))=M_{0}(f(0))$ and $M_{1}(f(t))=\rho$ for all $t\in [0,T)$.
	\end{corollary}
	The global existence of a solution to \eqref{infode}--\eqref{infIC}, particularly, depends on the growth condition $K_{j,k}\leq C(j^{\mu}k^{\nu}+j^{\nu}k^{\mu})$ $(\mu+\nu\leq3,$ $\mu,\nu\leq2)$ stated in the \Cref{mainthm-1}. This growth condition is consistent with findings from physical studies done in  \cite{Naim} under similar assumptions about interaction kernels of specific forms. However, for symmetric interaction kernels with the growth rates exceeding the ones specified in \eqref{quadraticcoagassum}, the finite time existence of solutions is shown through the following theorem.
	\begin{theorem}\label{localsolution}
	Let $K_{j,k}$ be the symmetric interaction kernel satisfying $K_{j,k}\leq Cj^{2}k^{2}$ for $(j,k)\in N_{0}\times N_{0}$ and $C>0$. If $f(0)=(f_{j}(0))_{j\ge 0} \in Y_{2}^{+}$ and $M_{2}(f(0))>0$, then the system  \eqref{infode}--\eqref{infIC} has a continuously differentiable solution $( f_{j})_{j\ge 0}\in  Y_{2}^{+}$  to \eqref{infode}--\eqref{infIC} satisfying  $M_{0}(f(t))=M_{0}(f(0))$ and $M_{1}(f(t))=M_{1}(f(0))$ for all $t\in [0,T_{0}]$ with $T_{0}<\left(2M_{2}(f(0))C\right)^{-1}$. 
	\end{theorem}
\begin{remark}
	\Cref{mainthm-1} and \Cref{localsolution} clearly extend the global and local existence results established in \cite[Theorem 2 and Corollary 3]{Esenturk}, respectively.
\end{remark}
	The work done in the physics literature \cite{Naim} suggest that $\eta = 3/2$ is the critical exponent beyond which the gelation phenomenon occurs, and occurrence of instantaneous gelation takes place for interaction kernels growing super quadratically. In the next two results, specifically in \Cref{finitetimegel} and \Cref{Tgel}, we have mathematically established the occurrence of finite-time gelation and instantaneous gelation to some extent, which were open problems mentioned in \cite[Section 4]{Esenturk}. 
	\begin{theorem}\label{finitetimegel}(Finite time gelation)
	 Suppose that the symmetric interaction kernel $K_{j,k}$ satisfies $C_{1}\left(j^{2}k^{\alpha}+j^{\alpha}k^{2}\right)\leq K_{j,k}\leq Cj^{2}k^{2}$ for all $(j,k)\in N_{0}\times N_{0}$, where $1<\alpha\leq 2$ and $C, C_{1}>0$.
		Assume also that $M_{\alpha}(f(0))>0$ and  $f(0)=(f_{j}(0))_{j\ge 0} \in Y_{2+\alpha}^{+}$. Then, the gelation time $T_{gel}$ for the solution  $( f_{j})_{j\ge 0}\in  Y_{2}^{+}$ to \eqref{infode}--\eqref{infIC}, constructed in \Cref{localsolution}, is finite. Additionally, for the specific class of interaction kernels  $ K_{j,k} = C j^2 k^2  $ for all  $ (j, k) \in \mathbb{N}_0 \times \mathbb{N}_0  $ and $C>0$, the gelation time can be expressed as  $ T_{\text{gel}} =\left(2M_{2}(f(0))C\right)^{-1}$.	
	\end{theorem}
The proof of \Cref{finitetimegel} utilizes some higher moment bounds and Gronwall's integral inequality. The main difficulty in this proof lies in the selection of a suitable moment that blows up in finite time; specifically, in this case, that moment is $M_{\alpha}(f)$ for some $\alpha > 1$. We have demonstrated the existence of a local solution in \Cref{localsolution} for interaction kernels satisfying $K_{j,k}\leq Cj^{2}k^{2}$. The subsequent \Cref{nonexistenceresult-1} reveals that the local solution cannot be extended to a global solution when the interaction kernel has a faster-growing lower bound.
	\begin{corollary}[Non-existence of global solution]\label{nonexistenceresult-1}
		Consider the infinite EDG system \eqref{infode}--\eqref{infIC}. Suppose that the symmetric interaction kernel $K_{j,k}$ satisfies
		\begin{align*}
			C_{1}\left(j^{2}k^{\alpha}+j^{\alpha}k^{2}\right)\leq K_{j,k}\leq Cj^{2}k^{2},\quad\text{for} \; (j,k)\in \mathbb{N}_{0} \times \mathbb{N}_{0},
		\end{align*}  
		where $1<\alpha\leq 2$ and $C, C_{1}>0.$
	Further, assume that  $ M_{\alpha}(f(0))> 0$ and  $f(0)=(f_{j}(0))_{j\ge 0}\in Y_{2+\alpha}^{+}$. Then, there is no global mass conserving solution $ \left(f_{j}\right)_{j\ge 0}\in  Y_{2}^{+}$ on $[0,\infty)$ to \eqref{infode}--\eqref{infIC}. 
	\end{corollary}
	\begin{theorem}\label{Tgel}(Instantaneous gelation)
		Let us consider the EDG system \eqref{infode}--\eqref{infIC}. Let $K_{j,k}$ be a non-negative symmetric interaction kernel satisfying $K_{j,0}=0$ for all $j \in \mathbb{N}_{0}$, with $K_{j,k}\geq C\left(j^{\beta}+k^{\beta}\right)$ for some $C>0$ and $\beta>2$.
		Assume also that $M_{0}(f(0))\geq C_{2}>0$ and $f(0) \in Y_{n}^{+}$ for all $n\in \mathbb{N}$. Then, the gelation time for any solution $\left(f_{j}\right)_{j\ge 0}\in  Y_{2}^{+}$ to \eqref{infode}--\eqref{infIC} is given by $T_{gel}=0$.
	\end{theorem}
	The proof of \Cref{Tgel} relies on the finiteness of all the higher moments of the initial data $f(0)=\left(f_{j}(0)\right)_{j\ge 0}$. With the aid of these finite higher moments of the initial data, we establish the finiteness of all higher moments of the solution to \eqref{infode}--\eqref{infIC}. Then, the proof of \Cref{Tgel} is given by Jensen's inequality and Gronwall's integral inequality.
	\begin{corollary}[Non-existence result]\label{nonexistenceresult-2}
		Let $K_{j,k}$ be a non-negative symmetric interaction kernel satisfying $K_{j,0}=0$ for all $j \in \mathbb{N}_{0}$, with $K_{j,k}\geq C\left(j^{\beta}+k^{\beta}\right)$ for some $C>0$ and $\beta>2$. Additionally, for every $n\in \mathbb{N}$, assume $M_{0}(f(0))\geq C_{2}>0$ and $f(0) \in Y_{n}^{+}$. Then
		there is no solution $ \left(f_{j}\right)_{j\ge 0}\in  Y_{2}^{+}$ to \eqref{infode}--\eqref{infIC} on any interval $[0,T)$ for $T>0$.
	\end{corollary}
	\section{Truncated Problem}
	\noindent Following the approach of \cite{Esenturk, ali2024well}, the proof of \Cref{mainthm-1} involves taking the limit of solutions derived from finite-dimensional systems of ordinary differential equations obtained by truncating the infinite EDG system \eqref{infode}--\eqref{infIC}. Specifically, for $N\ge 2$, we examine the following truncated system consisting of $N+1$ ordinary differential equations
	\begin{align}
		\text{\ }\dot{ f}_{0}^{N}&= f_{1}^{N}\sum_{k=0}^{N-1}K_{1,k} f_{k}^{N}- f_{0}%
		^{N}\sum_{k=1}^{N}K_{k,0} f_{k}^{N}, \label{Tode0}\\
		\text{\ }\dot{ f}_{j}^{N}  &  = f_{j+1}^{N}\sum_{k=0}^{N-1}K_{j+1,k} f_{k}%
		^{N}- f_{j}^{N}\sum_{k=0}^{N-1}K_{j,k} f_{k}^{N} \label{Todej}\\
		&\quad - f_{j}^{N}\sum_{k=1}^{N}K_{k,j} f_{k}^{N}+ f_{j-1}^{N}\sum_{k=1}%
		^{N}K_{k,j-1} f_{k}^{N},\quad \text{ for}\; 1\leq j\leq N-1 \nonumber\\
		\text{and} \quad \dot{ f}_{N}^{N}&=- f_{N}^{N}\sum_{k=0}^{N-1}K_{N,k} f_{k}^{N}+ f_{N-1}^{N}%
		\sum_{k=1}^{N}K_{k,N-1} f_{k}^{N}, \label{TodeN}
	\end{align}
	with 
	\begin{equation}
		f_{j}^{N}(0)= f_{j}(0)\geq0, \; \text{ for}\;0\leq j\leq N. \label{TodeIC}
	\end{equation}
	It is evident from the classical theory of ordinary differential equations, specifically from Picard's theorem, that the truncated system \eqref{Tode0}--\eqref{TodeIC} has a unique, continuously differentiable local solution. Proceeding as in \cite{Esenturk}, we present the following lemma, which will be used to establish the global existence of solutions to \eqref{Tode0}--\eqref{TodeIC} and for future reference.
	\begin{lemma}\label{L-divergenceform}
		\bigskip Let $ (h_{j})_{j=0}^{N}\in \mathbb{R}^{N+1}$ be a set of real numbers with $h_{j}\geq 0$ and $N \ge 2$.
		If $K_{j,k}$ is the symmetric interaction kernel, then we have
		\begin{align}
			\sum_{j=0}^{N} h_{j}\frac{d f_{j}^{N}}{dt}  &  =\sum_{j=1}^{N-1}( h_{j+1}%
			-2 h_{j}+ h_{j-1}) f_{j}^{N}\sum_{k=1}^{N-1}K_{j,k} f_{k}^{N}\nonumber \\
			&  +\sum_{j=1}^{N-1}\left(  ( h_{j-1}- h_{j})+( h_{1}- h_{0})\right)
			K_{j,0} f_{j}^{N} f_{0}^{N}\nonumber\\
			&  +\sum_{j=1}^{N-1}(( h_{j+1}- h_{j})+( h_{N-1}- h_{N})) f_{j}^{N}K_{N,j} f_{N}%
			^{N}\nonumber\\
			&  +(( h_{N-1}- h_{N})+( h_{1}- h_{0})) f_{N}^{N}K_{N,0} f_{0}^{N}.\label{mom-red2}
		\end{align}
	\end{lemma}
\noindent The proof of \Cref{L-divergenceform} can be found in \cite[Lemma 1]{Esenturk}. Next, we collect the global existence of non-negative solutions result to \eqref{Tode0}--\eqref{TodeIC}. We refer \cite[Lemma 2 and Corollary 1]{Esenturk} for the existence of a unique non-negative global solution to \eqref{Tode0}--\eqref{TodeIC}, which preserves the truncated zeroth and first moment of the truncated solution $\left(f_{j}^{N}(t)\right)_{j=0}^{N}$, i.e., 
\begin{align}\label{localparticleconservation}
			\sum_{j=0}^{N}{ f}_{j}^{N}(t)=	\sum_{j=0}^{N}{ f}_{j}^{N}(0)=\sum_{j=0}^{N}{ f}_{j}(0)\leq M_{0}(f(0)) 
		\end{align}
		and
		\begin{align}\label{localmassconservation}
			\sum_{j=0}^{N}j{ f}_{j}^{N}(t)=	\sum_{j=0}^{N}j{ f}_{j}^{N}(0)=\sum_{j=0}^{N}j{ f}_{j}(0)\leq M_{1}(f(0)),
		\end{align}
	for all $t\in [0, \infty)$. For future use, we extend the finite sequence of solutions  $(f_{j}^{N})_{j=0}^{N} $ to an infinite sequence  $(f_{j}^{N})_{j=0}^{\infty} $ by defining  $f_{j}^{N}(t) := 0 $ for all  $t \in [0, \infty)$ when  $j > N $.
	\begin{lemma}\label{truncatedlemma-1} 	Consider $T\in (0,\infty)$ and assume that the symmetric interaction kernel $K_{j,k}$ satisfies \eqref{quadraticcoagassum}. Let $f(0)\in Y_{1}^{+}$. Then there exist constants $C_{M}, C>0$ depending only on $C_{q}$ and $ f(0)$ such that 
		\begin{enumerate}
			\item If  $f(0) \in Y_{\lambda}^{+}$, then
			\begin{align}\label{momentbound-12} 	
				\sum_{j=0}^{N}j^{\lambda}{ f}_{j}^{N}(t)&\leq C_{M}e^{CT}, \quad \; \text{for each} \; t \in [0,T]
			\end{align}
			and
			\begin{align}\label{dmomentbound-12}
				\left|\dot{{ f}}_{j}^{N}(t)\right|&\leq 8C_{q}C_{M}^{2}e^{2CT},\quad \text{for}\quad 0\leq j \leq N, \quad \; \text{for each} \; t \in [0,T],
			\end{align}
			with 
			\begin{align*}
				\lambda:= \max\{\mu,\nu\}>1.
			\end{align*}
			\item If $\max\{\mu,\nu\} \leq 1$, then
			\begin{align}
				\left|\dot{{ f}}_{j}^{N}(t)\right|&\leq C_{q}M_{1}^{2}(f(0)), \quad \text{for}\quad 0\leq j \leq N, \quad \; \text{for each} \; t \in [0,T].\label{dmomentbound-3}
			\end{align}
		\end{enumerate}
	\end{lemma}
	\begin{proof}
		We begin with the proofs of \eqref{momentbound-12} and  \eqref{dmomentbound-12}. Define $h(x):=x^{\lambda}$ for all $x\in [0,\infty)$  and $h_{j}:=h(j)=j^{\lambda}$ for $0\leq j\leq N$. Then, $h^{\prime}(x)=\lambda x^{\lambda-1}$ and $h_{j}^{\prime}=h^{\prime}(j)$. First, we proceed to show that the second term on the right-hand side of \eqref{mom-red2} is non-positive. For $2\leq j\leq N-1,$ we have 
		\begin{equation*}
		(h_{j-1}- h_{j})+( h_{1}- h_{0})\leq  h_{1}^{\prime}- h_{j-1}^{\prime}\leq 0,
		\end{equation*}
		since $ h^{\prime}$ is increasing. For $j=1$, we obtain the following equation
		\begin{equation*}
		(h_{0}- h_{1})+( h_{1}- h_{0})= 0.
		\end{equation*}
		Similarly, the third and fourth terms on the right-hand side of \eqref{mom-red2} are non-positive since
		\begin{equation*}
		(h_{j+1}- h_{j})+( h_{N-1}- h_{N})\leq  h_{j+1}^{\prime}- h_{N-1}^{\prime}\leq0,\quad \text{for}\;\; 0\leq j \leq N-2,
		\end{equation*}
	  \begin{equation*}
		( h_{N}- h_{N-1})+( h_{N-1}- h_{N})=0,\quad \text{for}\;\; j=N-1,
	\end{equation*}
		and
		\begin{equation*}
		( h_{1}- h_{0})-( h_{N}- h_{N-1})=   h_{1}^{\prime}- h_{N-1}^{\prime}\leq 0,\quad \text{for}\; N\geq 2.
		\end{equation*}
		Then, from \eqref{mom-red2}, we deduce the following inequality 
		\begin{align}\label{lambda-0}
			\frac{d}{dt}\left(\sum_{j=0}^{N}h_{j} f_{j}^{N}\right)\leq \sum_{j=1}^{N-1}\left( h_{j+1}
			-2 h_{j}+ h_{j-1}\right) f_{j}^{N}\sum_{k=1}^{N-1}K_{j,k} f_{k}^{N}.
		\end{align}
		An application of the mean value theorem gives
		\begin{align}\label{mean value theorem} 
			h_{j+1}-2 h_{j}+ h_{j-1}=h_{j}^{\prime \prime }(\theta(j) )=\lambda(\lambda-1)\theta(j)^{\lambda-2}\quad \text{for some}\quad  \theta(j)\in (j-1,j+1).
		\end{align}
	Since $1<\lambda\leq 2$, we get
		\begin{align} \label{mean value theorem-1}
		h_{j+1}-2 h_{j}+ h_{j-1}=h_{j}^{\prime \prime }(\theta(j) )\leq 2^{2-\lambda} \lambda(\lambda-1)j^{\lambda-2}\quad \text{for}\quad j\ge 2.
	\end{align}
For $j=1$, we use $h_{j}=j^{\lambda}$ to obtain
		\begin{align}\label{lambda-2}
			h_{2}
			-2 h_{1}+ h_{0}=2^{\lambda}-2.
		\end{align}
		Inserting \eqref{mean value theorem-1} and \eqref{lambda-2} into \eqref{lambda-0}, we infer that
		\begin{align}\label{tempmomentbound-1}
			\frac{d}{dt}\left(\sum_{j=0}^{N}j^{\lambda} f_{j}^{N}\right)\leq C_{\lambda}\left(\sum_{k=1}^{N-1}K_{1,k} f_{1}^{N}f_{k}^{N} +\sum_{j=2}^{N-1} \sum_{k=1}^{N-1} j^{\lambda-2}K_{j,k} f_{j}^{N}f_{k}^{N}\right),
		\end{align}
		where $C_{\lambda}:=\max\{2^{\lambda}-2, 2^{2-\lambda}\lambda (\lambda-1)\}$. Next, using \eqref{quadraticcoagassum} and \eqref{localmassconservation} in \eqref{tempmomentbound-1}, we obtain   
		\begin{align}
			\frac{d}{dt}\left(\sum_{j=0}^{N}j^{\lambda} f_{j}^{N}\right)&\leq 2C_{\lambda}C_{q}M_{1}(f(0))\sum_{j=1}^{N-1} j^{\lambda}f_{j}^{N}\nonumber \\ &\quad +C_{\lambda}C_{q}\sum_{j=2}^{N-1} \sum_{k=1}^{N-1} j^{\lambda-2}\left(j^{\mu}k^{\nu} + j^{\nu}k^{\mu} \right) f_{j}^{N}f_{k}^{N}.\label{tempmomentbound-2}
		\end{align}
	Without loss of generality, we can assume that $\mu \ge \nu$. It gives $\lambda=\mu\geq \nu$. Since $\mu + \nu\le 3$, we have the following 
	   	\begin{align}\label{tempmomentbound-3}
	   		j^{\lambda-2}\left(j^{\mu}k^{\nu} + j^{\nu}k^{\mu} \right)&=j^{2\lambda-2} k^{\nu}+ j^{\mu +\nu-2}k^{\nu}\leq j^{2\lambda-2} k^{\nu}+ jk^{\nu}.  
	   	\end{align}
	First, we consider the case when $\nu \leq 1$ or $\lambda \leq 3/2$. In this case, from \eqref{tempmomentbound-3}, we deduce the following estimate 	
		\begin{align}\label{tempmomentbound-4} 
			j^{\lambda-2}\left(j^{\mu}k^{\nu} + j^{\nu}k^{\mu} \right)\leq j^{\lambda}k + jk^{\lambda}.
		\end{align}
Therefore, form \eqref{localmassconservation} and \eqref{tempmomentbound-4}, we have
\begin{align}\label{tempmomentbound-5}
\sum_{j=2}^{N-1} \sum_{k=1}^{N-1} j^{\lambda-2}\left(j^{\mu}k^{\nu} + j^{\nu}k^{\mu} \right) f_{j}^{N}f_{k}^{N} \leq 2 M_{1}(f(0)) \sum_{j=1}^{N-1} j^{\lambda}f_{j}^{N}. 	
\end{align}
Next, we consider the remaining case, i.e., $\nu >1$ and $\lambda >3/2$. From \eqref{tempmomentbound-3}, we obtain
\begin{align}\label{tempmomentbound-6}
	\sum_{j=2}^{N-1} \sum_{k=1}^{N-1} j^{\lambda-2}\left(j^{\mu}k^{\nu} + j^{\nu}k^{\mu} \right) f_{j}^{N}f_{k}^{N} \leq \sum_{j=1}^{N-1} \sum_{k=1}^{N-1}\left( j^{2\lambda-2} k^{\nu}+ jk^{\lambda}\right)f_{j}^{N}f_{k}^{N}.
\end{align} 		
Applying H\"older's inequality and \eqref{localmassconservation} in \eqref{tempmomentbound-6}, we get
\begin{align}
	\sum_{j=2}^{N-1} \sum_{k=1}^{N-1} j^{\lambda-2}\left(j^{\mu}k^{\nu} + j^{\nu}k^{\mu} \right) f_{j}^{N}f_{k}^{N}&\leq \left(\sum_{j=1}^{N-1}jf_{j}^{N}\right)^{\frac{2-\nu}{\lambda-1}} \left(\sum_{j=1}^{N-1}j^{\lambda}f_{j}^{N}\right)^{\frac{2\lambda+\nu-4}{\lambda-1}}\nonumber\\
	&\quad + M_{1}(f(0))\sum_{j=1}^{N-1} j^{\lambda}f_{j}^{N}\nonumber\\
	&\leq C_{L}\left[ \left(\sum_{j=1}^{N-1}j^{\lambda}f_{j}^{N}\right)^{\frac{2\lambda+\nu-4}{\lambda-1}}+\sum_{j=1}^{N-1} j^{\lambda}f_{j}^{N}\right], \label{tempmomentbound-7}
\end{align}
where $C_{L}:= \max \left\{{M_{1}(f(0))}^{\frac{2-\nu}{\lambda-1}}, M_{1}(f(0))\right\}$. Now, using Young's inequality in \eqref{tempmomentbound-7}, we deduce the following estimate
\begin{align}
	\sum_{j=2}^{N-1} \sum_{k=1}^{N-1} j^{\lambda-2}\left(j^{\mu}k^{\nu} + j^{\nu}k^{\mu} \right) f_{j}^{N}f_{k}^{N} &\leq C_{L}\left[1+2\sum_{j=1}^{N-1} j^{\lambda}f_{j}^{N}\right].\label{tempmomentbound-8}
\end{align}	
From \eqref{tempmomentbound-5} and \eqref{tempmomentbound-8}, we infer that
\begin{align}
	\sum_{j=2}^{N-1} \sum_{k=1}^{N-1} j^{\lambda-2}\left(j^{\mu}k^{\nu} + j^{\nu}k^{\mu} \right) f_{j}^{N}f_{k}^{N} &\leq 2C_{L}\left[1+\sum_{j=1}^{N-1} j^{\lambda}f_{j}^{N}\right],\label{tempmomentbound-9}
\end{align}
for any $\mu$ and $\nu$	satisfying the conditions mentioned in \eqref{quadraticcoagassum}. Hence, using \eqref{tempmomentbound-9} in \eqref{tempmomentbound-2}, we get		
	\begin{align}
		\frac{d}{dt}\left(\sum_{j=0}^{N}j^{\lambda} f_{j}^{N}\right)&\leq 2C_{\lambda}C_{q}M_{1}(f(0))\sum_{j=1}^{N-1} j^{\lambda}f_{j}^{N}+ 2C_{L}C_{\lambda}C_{q}\left[1+\sum_{j=1}^{N-1} j^{\lambda}f_{j}^{N}\right],\nonumber\\
		&\leq 2C_{L}C_{\lambda}C_{q}+\left[2C_{\lambda}C_{q}M_{1}(f(0)) + 2C_{L}C_{\lambda}C_{q}\right]\sum_{j=0}^{N} j^{\lambda}f_{j}^{N}.\nonumber 
	\end{align}	
Solving the above differential inequality, we conclude \eqref{momentbound-12}
		with $C=2C_{\lambda}C_{q}M_{1}(f(0)) + 2C_{L}C_{\lambda}C_{q}$ and $C_{M}=2C_{L}C_{\lambda}C_{q}+M_{\lambda}(f(0))$. Next, we proceed to prove \eqref{dmomentbound-12}. For $\lambda > 1$, by using the non-negativity of the interaction kernel $K_{j,k}$ and the truncated solution $\left(f_{j}^{N}(t)\right)_{j=0}^{N}$, and by applying \eqref{quadraticcoagassum} and \eqref{momentbound-12} in \eqref{Tode0}, we obtain
		\begin{align}
			|\dot{ f}_{0}^{N}|&\leq f_{1}^{N}\sum_{k=0}^{N-1}K_{1,k} f_{k}^{N}+f_{0}
			^{N}\sum_{k=1}^{N}K_{k,0} f_{k}^{N}\nonumber\\
			&\leq 2C_{q}f_{1}^{N}\sum_{k=0}^{N-1}k^{\lambda}  f_{k}^{N}\nonumber\\
			&\leq 2 C_{q}C_{M}^{2}e^{2CT}.\label{1dtemp0}
		\end{align}
		For $1\leq j\leq N-1$, from \eqref{Todej},\eqref{quadraticcoagassum}, and \eqref{momentbound-12}, we derive the following estimate
		\begin{align}
			|\dot{ f}_{j}^{N}|  &  \leq  f_{j+1}^{N}\sum_{k=0}^{N-1}K_{j+1,k} f_{k}%
			^{N}+ f_{j}^{N}\sum_{k=0}^{N-1}K_{j,k} f_{k}^{N} \nonumber\\
			&\quad +f_{j}^{N}\sum_{k=1}^{N}K_{k,j} f_{k}^{N}+ f_{j-1}^{N}\sum_{k=1}%
			^{N}K_{k,j-1} f_{k}^{N}\nonumber\\
			&  \leq  2C_{q}(j+1)^{\lambda}f_{j+1}^{N}\sum_{k=0}^{N-1}k^{\lambda}f_{k}^{N}+ 2C_{q}j^{\lambda}f_{j}^{N}\sum_{k=0}^{N-1} k^{\lambda} f_{k}^{N} \nonumber \\
			&\quad +2C_{q}j^{\lambda}f_{j}^{N}\sum_{k=1}^{N}k^{\lambda} f_{k}^{N}+ 2C_{q}(j-1)^{\lambda}f_{j-1}^{N}\sum_{k=1}^{N}k^{\lambda} f_{k}^{N}\nonumber\\
			&\leq 8 C_{q}C_{M}^{2}e^{2CT}.\label{1dtempj}
		\end{align}
		Similarly, from \eqref{TodeN},\eqref{quadraticcoagassum}, and \eqref{momentbound-12}, we deduce 
		\begin{align}
			|\dot{ f}_{N}^{N}|&\leq f_{N}^{N}\sum_{k=0}^{N-1}K_{N,k} f_{k}^{N}+ f_{N-1}^{N}%
			\sum_{k=1}^{N}K_{k,N-1} f_{k}^{N}\nonumber\\
			&\leq 2C_{q}N^{\lambda}f_{N}^{N}\sum_{k=0}^{N-1}k^{\lambda}f_{k}^{N}+ 2C_{q}(N-1)^{\lambda}f_{N-1}^{N}%
			\sum_{k=1}^{N}k^{\lambda}f_{k}^{N}\nonumber\\
			&\leq 4 C_{q}C_{M}^{2}e^{2CT}.\label{1dtempN}
		\end{align}
		Collecting the estimates from \eqref{1dtemp0}, \eqref{1dtempj} and \eqref{1dtempN}, we conclude \eqref{dmomentbound-12}. For $ \max\{\mu,\nu\} \leq 1$, \eqref{dmomentbound-3} directly follows from \cite[Theorem 1]{Esenturk}.	
	\end{proof}
	\Cref{truncatedlemma-1} implies that the sequence $( f_{j}^{N})_{j\ge 0}$ is uniformly bounded and equicontinuous. Then, by the Arzel\'{a}-Ascoli theorem, there exists a subsequence $(f_{j}^{N})_{j\ge 0}$ (not relabeled) that converges uniformly to a continuous function, say $ (f_{j})_{j\ge 0}$. Next, we have $M_{\lambda}(f(t))<\infty$, for all $t\in [0,T]$, which follows from the uniform convergence of $(f_{j}^{N})_{j\ge 0}$ to $(f_{j})_{j\ge 0}$ on $[0,T]$ and \eqref{momentbound-12}, as demonstrated below
	\begin{align*}
		M_{\lambda}(f(t))=	\sum_{j=0}^{\infty} j^{\lambda}f_{j}(t)=\lim_{M\rightarrow\infty}\sum_{j=0}^{M} j^{\lambda}f_{j}(t)=\lim_{M\rightarrow\infty}\lim_{N \rightarrow\infty}\sum_{j=0}^{M} j^{\lambda}f_{j}^{N}(t)\leq M_{\lambda}(f(0))e^{CT}.	
	\end{align*}
	Similarly, from \eqref{localparticleconservation} and \eqref{localmassconservation}, we have
	\begin{align}\label{solparticlemassbound}
		M_{0}(f(t))\leq M_{0}(f(0))\quad \text{and}\quad 	M_{1}(f(t))\leq M_{1}(f(0)),\; \text{for each}\; t\in [0,T].	
	\end{align}
	Next, we show that $ (f_{j})_{j\ge 0}$ is a mild solution to the original problem \eqref{infode}--\eqref{infIC} on $[0,T]$ by proving that the series $\sum_{k=1}^{N}K_{j,k} f_{k}^{N}$ converges uniformly on every bounded intervals of time $[0,T].$ To prove this, we need the boundedness of higher moments, as presented in \Cref{Tail}. Therefore, we introduce the convex function technique in the next part to control the large size. We denote by $\mathcal{G}_{1,\infty}$ the set of non-negative, convex functions $G\in C^{2}([0,\infty))$ such that $G(0) = G^{\prime}
	(0)=0$ and $G^{\prime}$ is a concave function satisfying $G^{\prime}(x) >0$ for $x>0$ with
	\begin{align}
	\lim_{x\rightarrow \infty} G^{\prime}(x)=\lim_{x\rightarrow \infty} \frac{G(x)}{x}=\infty.\label{vanishingconvexlimit}
	\end{align}
	It is clear that $x\mapsto x^{p}$ belongs to  $\mathcal{G}_{1,\infty}$ if $p\in (1, 2].$ We are now in a position to find the higher moment bounds of the initial data $f(0)=(f_{j,0})_{j\ge 0}$. For this purpose, we recall a consequence of a refined version of the de la Vall\'{e}e-Poussin theorem for integrable functions \cite[Proposition I.1.1]{Le1977} and \cite[Theorem 8]{laurenccot2014weak}.
	\begin{lemma}\label{Delavalle}
		Let $f(0)=(f_{j,0})_{j\ge 0}$ be the initial data to the problem \eqref{infode}--\eqref{infIC}. Then, the following results hold
	\begin{enumerate}[label=(\alph*)]
			\item If $f(0)=(f_{j,0})_{j\ge 0}\ \in Y_{\lambda}^{+}$, then there exists a function $G_{\lambda}\in \mathcal{G}_{1,\infty}$ such that
			\begin{align}
				\sum_{j=0}^{\infty}j^{\lambda-1}G_{\lambda}(j) f_{j}(0)<\infty. \nonumber
			\end{align}
			\item If $f(0)=(f_{j,0})_{j\ge 0}\ \in Y_{1}^{+}$, then there exists a function $G_{1}\in \mathcal{G}_{1,\infty}$ such that
			\begin{align}
				\sum_{j=0}^{\infty} G_{1}(j) f_{j}(0)<\infty. \nonumber
			\end{align}
		\end{enumerate}
	\end{lemma}
	We collect the following properties of the function $G \in \mathcal{G}_{1,\infty}$ for future use.
	\begin{lemma}
		For any $G \in \mathcal{G}_{1,\infty}$, we have the following inequalities
		\begin{align}
			0&\leq G(x)\leq xG^{\prime}(x)\leq 2G(x),\label{firstdiffinequality}\\
			0&\leq G(rx)\leq \max \left\{1,r^{2}\right\}G(x),\label{scalinginequality}\\
		\text{and}\quad	0&\leq xG^{\prime\prime}(x)\leq G^{\prime}(x),\label{seconddiffinequality}
		\end{align}
		for $x\geq 0$ and $r\geq 0$.
		Moreover, for $x\geq 2$ and $\lambda \in (1,2]$, we obtain 
		\begin{align}
			(x+1)^{\lambda-1}G(x+1)-2x^{\lambda-1}G(x)+(x-1)^{\lambda-1}G(x-1)&\leq \left(4\lambda-2\right) G^{\prime}(x+1),\label{divconvexinequality-12}
		\end{align}
		and for $x\ge 1$,
		\begin{align}
			G(x+1)-2G(x)+G(x-1)&\leq G^{\prime}(x+1).\label{divconvexinequality-3}
		\end{align}
	\end{lemma}
	\begin{proof}
		The proofs of the inequalities \eqref{firstdiffinequality}--\eqref{seconddiffinequality} can be found in \cite[Proposition 2.14]{laurenccot2014weak} and \cite[Appendix]{laurenccotlifshitz}. Therefore, we focus on proving the inequalities \eqref{divconvexinequality-12} and \eqref{divconvexinequality-3}. First, let us prove \eqref{divconvexinequality-12}. Consider the following expression for $x\geq 2$,
		\begin{align}
			&(x+1)^{\lambda-1}G(x+1)-2x^{\lambda-1}G(x)+(x-1)^{\lambda-1}G(x-1) \nonumber\\
			&= \left[(x+1)^{\lambda-1}G(x+1)-x^{\lambda-1}G(x)\right]-\left[x^{\lambda-1}G(x)-(x-1)^{\lambda-1}G(x-1)\right] \nonumber\\
			&= \int_{x}^{x+1} (y^{\lambda-1}G(y))^{\prime} \, dy - \int_{x-1}^{x} (y^{\lambda-1}G(y))^{\prime} \, dy. \label{1div-1}
		\end{align}
		Clearly,  $ y \mapsto (y^{\lambda-1}G(y))^{\prime}  $ is an increasing function on $(0,\infty)$, which can be shown in the following way
		\begin{align}\label{doublediffG-1}
			(y^{\lambda-1}G(y))^{\prime\prime}&=\left(\lambda-1\right)\left(\lambda-2\right)y^{\lambda-3}G(y)+2\left(\lambda-1\right)y^{\lambda-2}G^{\prime}(y)+y^{\lambda-1}G^{\prime\prime}(y).
		\end{align}
	Since $\lambda >1$, using \eqref{firstdiffinequality} and \eqref{seconddiffinequality} in \eqref{doublediffG-1}, we conclude that
		\begin{align*}
			(y^{\lambda-1}G(y))^{\prime\prime}&\geq \lambda \left(\lambda-1\right)y^{\lambda-3}G(y)+y^{\lambda-1}G^{\prime\prime}(y)\geq 0.
		\end{align*}
		Therefore, by using the monotonicity properties of the functions $y \mapsto (y^{\lambda-1}G(y))^{\prime}$ and  $ y \mapsto G^{\prime}(y)  $, along with \eqref{doublediffG-1}, \eqref{seconddiffinequality}, and the condition  $1< \lambda \leq 2 $ in \eqref{1div-1}, we obtain the following inequality
		\begin{align}
			&(x+1)^{\lambda-1}G(x+1)-2x^{\lambda-1}G(x)+(x-1)^{\lambda-1}G(x-1)\nonumber \\
			&\leq \left[(x+1)^{\lambda-1}G(x+1)\right]^{\prime}-\left[(x-1)^{\lambda-1}G(x-1)\right]^{\prime}\nonumber\\
			&=\int_{x-1}^{x+1}(y^{\lambda-1}G(y))^{\prime\prime} dy\nonumber\\
			&=\int_{x-1}^{x+1}\left[\left(\lambda-1\right)\left(\lambda-2\right)y^{\lambda-3}G(y)+2\left(\lambda-1\right)y^{\lambda-2}G^{\prime}(y)+y^{\lambda-1}G^{\prime\prime}(y)\right] dy\nonumber\\
			&\leq \int_{x-1}^{x+1} \left[ 2\left(\lambda-1\right)G^{\prime}(y)+G^{\prime}(y)\right] dy\leq \left(4\lambda-2\right) G^{\prime}(x+1),\nonumber
		\end{align}
		which completes the proof of \eqref{divconvexinequality-12}. Next, to prove \eqref{divconvexinequality-3}, we consider
		\begin{align}
			G(x+1)-2G(x)+G(x-1)&=	G(x+1)-G(x)-[G(x)-G(x-1)]\nonumber\\
			&=\int_{x}^{x+1}G^{\prime}(y)dy-\int_{x-1}^{x}G^{\prime}(y)dy,\nonumber
		\end{align}
	 for $x\geq 1$.
		From the monotonicity and the non-negativity of $G^{\prime}$, we infer that
		\begin{align}
			G(x+1)-2G(x)+G(x-1)&\leq G^{\prime}(x+1),\nonumber
		\end{align}
		which completes the proof of \eqref{divconvexinequality-3}.
	\end{proof}
	The following lemma shows the boundedness of the higher moments.
	\begin{lemma}\label{Tail}
	Assume that the interaction kernel $K_{j,k}$ is given by \eqref{quadraticcoagassum} and that $T\in (0, \infty)$. Suppose  $ G_{\lambda}, G_{1} \in \mathcal{G}_{1,\infty}$, satisfying \Cref{Delavalle}(a) and \Cref{Delavalle}(b), respectively with $\lambda$ defined as in \eqref{LambdaD}. Then, the following statements hold
		\begin{enumerate}
			\item If  $ f(0)=(f_{j,0})_{j\ge 0} \in Y_{\lambda}^{+}$, then there exists a positive constant $C_{1,2}(T)>0$ depending only on the initial data $ (f_{j,0})_{j\ge 0}$, $C_{q}$ and $T$ such that
			\begin{align}
				\sum_{j=0}^{N}j^{\lambda-1}G_{\lambda}(j) f_{j}^{N}(t)<C_{1,2}(T), \quad \text{for} \; t \in [0,T] \label{mainconvexinequality-lamda12}
			\end{align}
			and 
			\begin{align}
				\sum_{j=0}^{\infty}j^{\lambda-1}G_{\lambda}(j) f_{j}(t)<C_{1,2}(T), \quad \text{for} \; t \in [0,T]. \label{solmainconvexinequality-lambda12}
			\end{align}
			\item  If $ f(0)=(f_{j,0})_{j\ge 0} \in Y_{1}^{+}$ and $\max\{\mu,\nu\} \leq 1 \text{ and } \min\{\mu,\nu\} < 1$, then there exists a positive constant $C_{3}(T)>0$ depending only on initial data $ (f_{j,0})_{j\ge 0}$, $C_{q}$ and $T$ such that
			\begin{align}
				\sum_{j=0}^{N} G_{1}(j) f_{j}^{N}(t)<C_{3}(T), \quad \text{for} \; t \in [0,T] \label{mainconvexinequality-lamda3}
			\end{align}
			and 
			\begin{align}
				\sum_{j=0}^{\infty} G_{1}(j) f_{j}(t)<C_{3}(T), \quad \text{for} \;t \in [0,T]. \label{solmainconvexinequality-lambda3}
			\end{align}
		\end{enumerate}
	\end{lemma}
	\begin{proof} Let us first prove the estimate \eqref{mainconvexinequality-lamda12}. For this purpose, define
		\begin{align*}
			h(x)&:=x^{\lambda-1}G_{\lambda}(x),\; h^{\prime}(x):=(\lambda-1)x^{\lambda-2}G_{\lambda}(x)+x^{\lambda-1}G_{\lambda}^{\prime}(x),\quad \text{for}\; x\in (0,\infty),
		\end{align*}
	with
		\begin{align*} h_{j}&:=h(j)=j^{\lambda-1}G_{\lambda}(j),\; h_{j}^{\prime}:=h^{\prime}(j),\quad  j\in \mathbb{N} \;\text{ and}\;h_{0}:=0.  
		\end{align*}
		It is clear from the proof of \eqref{divconvexinequality-12} that  $ x \mapsto (x^{\lambda-1}G_{\lambda}(x))^{\prime}  $ is an increasing function on $(0,\infty)$. Therefore, we adopt a similar approach to that used in the proof of \Cref{truncatedlemma-1}(1). As a result, from \eqref{mom-red2}, we obtain
		\begin{align}\label{Tail-10}
			\frac{d}{dt}\left(\sum_{j=0}^{N}j^{\lambda-1}G_{\lambda}(j) f_{j}^{N}(t)\right)\leq \sum_{j=1}^{N-1}\left( h_{j+1}
			-2 h_{j}+ h_{j-1}\right) f_{j}^{N}(t)\sum_{k=1}^{N-1}K_{j,k} f_{k}^{N}(t).
		\end{align}
		Now, from \eqref{divconvexinequality-12} and \eqref{firstdiffinequality}, we have 
		\begin{align}
			h_{j+1}-2 h_{j}+ h_{j-1}&\leq \max \{4\lambda-2, 2^{\lambda-1}\lambda\}G_{\lambda}^{\prime}(j+1) \quad \text{for}\quad 1\leq j\leq N-1.\label{divconvexinequalitytail-1}
		\end{align}
		Inserting \eqref{divconvexinequalitytail-1} and \eqref{quadraticcoagassum} into \eqref{Tail-10}, we infer that
		\begin{align}
			\frac{d}{dt}\left(\sum_{j=0}^{N}j^{\lambda-1}G_{\lambda}(j) f_{j}^{N}(t)\right)&\leq \tilde{C} \sum_{j=1}^{N-1}\sum_{k=1}^{N-1}G_{\lambda}^{\prime}(j+1) j^{\lambda}k^{\lambda} f_{j}^{N}(t) f_{k}^{N}(t),
		\end{align}
		where $\tilde{C}:=2C_{q}\max \{4\lambda-2,2^{\lambda-1}\lambda\}$. Using \eqref{firstdiffinequality}, \eqref{momentbound-12}, \eqref{scalinginequality}, and noting that $\lambda >1$, we obtain
		\begin{align}
			\frac{d}{dt}\left(\sum_{j=0}^{N}j^{\lambda-1}G_{\lambda}(j) f_{j}^{N}(t)\right)&\leq 2\tilde{C} \sum_{j=1}^{N-1}\sum_{k=1}^{N-1}G_{\lambda}(j+1)j^{\lambda-1}k^{\lambda} f_{j}^{N}(t) f_{k}^{N}(t)\nonumber\\
			&\leq  2\tilde{C} C_{M}e^{CT}\sum_{j=1}^{N-1}j^{\lambda-1}G_{\lambda}(j+1) f_{j}^{N}(t)\nonumber\\
			&\leq  2\tilde{C} C_{M}e^{CT}\sum_{j=1}^{N-1}j^{\lambda-1}\left(\frac{j+1}{j}\right)^{2}G_{\lambda}(j) f_{j}^{N}(t)\nonumber\\
			&\leq 8\tilde{C} C_{M}e^{CT}\sum_{j=0}^{N}j^{\lambda-1}G_{\lambda}(j) f_{j}^{N}(t).\nonumber
		\end{align}
		Then, \eqref{mainconvexinequality-lamda12} follows from \Cref{Delavalle}(a) and Gronwall's differential inequality. Next, \eqref{solmainconvexinequality-lambda12} follows from the uniform convergence of $(f_{j}^{N})_{j\ge 0}$ to $(f_{j})_{j\ge 0}$ on $[0,T]$ and \eqref{mainconvexinequality-lamda12}, i.e.,
		\begin{align}
			\sum_{j=0}^{\infty}j^{\lambda-1}G_{\lambda}(j) f_{j}(t)=\lim_{M\rightarrow\infty}\sum_{j=0}^{M} j^{\lambda-1}G_{\lambda}(j) f_{j}(t)=\lim_{M\rightarrow\infty}\lim_{N\rightarrow\infty}\sum_{j=0}^{M} j^{\lambda-1}G_{\lambda}(j)f_{j}^{N}(t)\leq C_{1,2}(T).\nonumber	
		\end{align}
		Now, we proceed to prove \eqref{mainconvexinequality-lamda3}. From  \eqref{mom-red2} and the monotonicity of the function $x \mapsto G_{1}^{\prime}(x)$ on  $ (0, \infty)  $, we deduce the following inequality
		\begin{linenomath} 
		\begin{align}\label{Tail-30}
			\frac{d}{dt}\left(\sum_{j=0}^{N} G_{1}(j) f_{j}^{N}(t)\right)\leq \sum_{j=1}^{N-1}\left( G_{1}(j+1) - 2G_{1}(j) + G_{1}(j-1)\right) f_{j}^{N}(t)\sum_{k=1}^{N-1}K_{j,k} f_{k}^{N}(t).
		\end{align}
	\end{linenomath}
	From \eqref{divconvexinequality-3}, we can write
	\begin{linenomath}
		\begin{align}
			G_{1}(j+1) - 2G_{1}(j) + G_{1}(j-1) &\leq G_{1}^{\prime}(j+1) \quad \text{for} \quad 1 \leq j \leq N-1.\label{divconvexinequalitytail-3}
		\end{align}
	\end{linenomath}
Applying \eqref{divconvexinequalitytail-3}, \eqref{quadraticcoagassum}, and  using $\max\{\mu,\nu\} \leq 1$, \eqref{Tail-30} can be further estimated as
\begin{linenomath}
		\begin{align}
			\frac{d}{dt}\left(\sum_{j=0}^{N} G_{1}(j) f_{j}^{N}(t)\right) &\leq 2C_{q} \sum_{j=1}^{N-1}\sum_{k=1}^{N-1}G_{1}^{\prime}(j+1) jk f_{j}^{N}(t) f_{k}^{N}(t).
		\end{align}
	\end{linenomath}
		Using \eqref{firstdiffinequality}, \eqref{localmassconservation} and \eqref{scalinginequality}, we obtain
		\begin{linenomath}
		\begin{align}
			\frac{d}{dt}\left(\sum_{j=0}^{N} G_{1}(j) f_{j}^{N}(t)\right) &\leq  4C_{q} \sum_{j=1}^{N-1}\sum_{k=1}^{N-1}G_{1}(j+1)k f_{j}^{N}(t) f_{k}^{N}(t) \nonumber\\
			&\leq 4C_{q} M_{1}(f(0))\sum_{j=1}^{N-1}\left(\frac{j+1}{j}\right)^{2}G_{1}(j) f_{j}^{N}(t) \nonumber\\
			&\leq 16C_{q} M_{1}(f(0))\sum_{j=0}^{N} G_{1}(j) f_{j}^{N}(t). \nonumber
		\end{align}
	\end{linenomath}
		Now, the inequality \eqref{mainconvexinequality-lamda3} follows from \Cref{Delavalle}(b) and Gronwall's differential inequality. Subsequently, \eqref{solmainconvexinequality-lambda3} is obtained from the uniform convergence of  $(f_{j}^{N})_{j\ge 0}$ to $( f_{j})_{j\ge 0}$ on $[0,T]$ and \eqref{mainconvexinequality-lamda3}, i.e.,
		\begin{linenomath}
		\begin{align*}
			\sum_{j=0}^{\infty} G_{1}(j) f_{j}(t)=\lim_{M\rightarrow\infty}\sum_{j=0}^{M} G_{1}(j) f_{j}(t)=\lim_{M\rightarrow\infty}\lim_{N\rightarrow\infty}\sum_{j=0}^{M} G_{1}(j)f_{j}^{N}(t)\leq C_{3}(T),
		\end{align*}
	\end{linenomath}
		which completes the proof of \Cref{Tail}.	
	\end{proof}
	We are now in a position to complete the proof of \Cref{mainthm-1}.
	\begin{proof}[\textbf{Proof of \Cref{mainthm-1}.}]
		Let us proceed with the proof of \Cref{mainthm-1}. In order to show part(a), we first claim that $\sum_{k=0}^{N-1}K_{j+1,k} f_{k}^{N}$ converges uniformly to $\sum_{k=0}^{\infty}K_{j+1,k} f_{k}$ on $[0,T]$ for every $j\in \mathbb{N}_{0}$. We observe that for $M\in \mathbb{N}$,
		\begin{linenomath}
		\begin{equation}
			\left\vert \sum_{k=0}^{\infty}K_{j+1,k} f_{k}^{N}-\sum_{k=0}^{\infty}%
			K_{j+1,k} f_{k}\right\vert \leq\sum_{k=0}^{M}K_{j+1,k}\left\vert  f_{k}^{N}%
			- f_{k}\right\vert +\left\vert \sum_{k=M+1}^{\infty}K_{j+1,k}( f_{k}+ f_{k}%
			^{N})\right\vert, \label{unidif1}%
		\end{equation}
	\end{linenomath}
		where $f_{k}^{N}:=0$ for $k>N$. Since $\lambda >1$, substituting \eqref{quadraticcoagassum}, \eqref{mainconvexinequality-lamda12} and \eqref{solmainconvexinequality-lambda12} into the second term on the right-hand side of \eqref{unidif1} yields
		\begin{linenomath} 
		\begin{align}
			\left\vert \sum_{k=M+1}^{\infty}K_{j+1,k}( f_{k}+ f_{k}^{N})\right\vert
			&\leq 2C_{q}(j+1)^{\lambda}\sum_{k=M+1}^{\infty} k^{\lambda} ( f_{k}+ f_{k}^{N}) \nonumber\\
			&\leq 2C_{q}(j+1)^{\lambda}\sup_{k\geq M+1}\left(\frac{k}{G_{\lambda}(k)}\right)\sum_{k=M+1}^{\infty}k^{\lambda-1}G_{\lambda}(k)( f_{k}+f_{k}^{N})\nonumber\\
			&\leq C_{s}(T)\sup_{k\geq M+1}\left(\frac{k}{G_{\lambda}(k)}\right),\label{coagtermconvergence-1}
		\end{align}
	\end{linenomath}
where $C_{s}(T):= 4C_{q}(j+1)^{\lambda}C_{1,2}(T)$.
	From \eqref{coagtermconvergence-1}, we infer that
	\begin{linenomath}
	\begin{align}\label{kertailvanish}
		\sup_{t\in [0,T]} K_{j+1,N}f_{N}^{N}(t)\leq \sup_{t\in [0,T]} \sum_{k=N}^{\infty}K_{j+1,k}f_{k}^{N}(t)\leq C_{s}(T)\sup_{k\geq N}\left(\frac{k}{G_{\lambda}(k)}\right).
	\end{align}
\end{linenomath}
		By applying \eqref{coagtermconvergence-1} in \eqref{unidif1} and using the triangle inequality, we get
		\begin{linenomath}
		\begin{align*}
			\sup_{t\in [0,T]}	\left\vert \sum_{k=0}^{N-1}K_{j+1,k} f_{k}^{N}-\sum_{k=0}^{\infty}
			K_{j+1,k} f_{k}\right\vert &\leq 	\sup_{t\in [0,T]}\sum_{k=0}^{M}K_{j+1,k}\left\vert  f_{k}^{N}
			- f_{k}\right\vert +	\sup_{t\in [0,T]} K_{j+1,N}f_{N}^{N}\nonumber\\ &\quad +C_{s}(T)\sup_{k\geq M+1}\left(\frac{k}{G_{\lambda}(k)}\right).
		\end{align*}
	\end{linenomath}
		Applying \eqref{kertailvanish}, the above inequality can be further estimated as
		\begin{linenomath} 
		\begin{align*}
			\sup_{t\in [0,T]}	\left\vert \sum_{k=0}^{N-1}K_{j+1,k} f_{k}^{N}-\sum_{k=0}^{\infty}
			K_{j+1,k} f_{k}\right\vert &\leq 	\sup_{t\in [0,T]}\sum_{k=0}^{M}K_{j+1,k}\left\vert  f_{k}^{N}
			- f_{k}\right\vert +C_{s}(T)\sup_{k\geq N}\left(\frac{k}{G_{\lambda}(k)}\right)\nonumber\\ &\quad +C_{s}(T)\sup_{k\geq M+1}\left(\frac{k}{G_{\lambda}(k)}\right).
		\end{align*}
	\end{linenomath}
Taking the limit as $N\rightarrow \infty$, we use \eqref{vanishingconvexlimit} and the uniform convergence of $(f_{j}^{N})_{j\ge 0}$ to $( f_{j})_{j\ge 0}$ on $[0,T]$ to obtain the following inequality
\begin{linenomath}
		\begin{align*}
			\lim_{N\rightarrow\infty} 	\sup_{t\in [0,T]} \left\vert \sum_{k=0}^{N-1}K_{j+1,k} f_{k}^{N}-\sum_{k=0}^{\infty}
			K_{j+1,k} f_{k}\right\vert &\leq C_{s}(T)\sup_{k\geq M+1}\left(\frac{k}{G_{\lambda}(k)}\right).
		\end{align*}
	\end{linenomath}
		Finally, passing the limit as $M\rightarrow \infty$ and using \eqref{vanishingconvexlimit}, we conclude that
		\begin{linenomath}
		\begin{align}\label{KC1}
				\lim_{N\rightarrow\infty} 	\sup_{t\in [0,T]} \left\vert \sum_{k=0}^{N-1}K_{j+1,k} f_{k}^{N}(t)-\sum_{k=0}^{\infty}
			K_{j+1,k} f_{k}(t)\right\vert &=0.
		\end{align}
	\end{linenomath}
	 Similarly, since $\lambda >1$, by using \eqref{quadraticcoagassum}, \eqref{mainconvexinequality-lamda12}, \eqref{solmainconvexinequality-lambda12}, \eqref{vanishingconvexlimit}, and the uniform convergence of  $(f_{j}^{N})_{j \ge 0} $ to  $(f_{j})_{j \ge 0} $ on  $[0,T] $, we can readily infer the following uniform convergence
	 \begin{linenomath}
		\begin{align}
			\lim_{N\rightarrow\infty} 	\sup_{t\in [0,T]} \left\vert \sum_{k=0}^{N-1}K_{j,k} f_{k}^{N}(t)-\sum_{k=0}^{\infty}
			K_{j,k} f_{k}(t)\right\vert &=0,\label{KC2} \\
				\lim_{N\rightarrow\infty} 	\sup_{t\in [0,T]} \left\vert \sum_{k=1}^{N}K_{k,j} f_{k}^{N}(t)-\sum_{k=1}^{\infty}
			K_{k,j} f_{k}(t)\right\vert &=0,\label{KC3}\\
				\lim_{N\rightarrow\infty} 	\sup_{t\in [0,T]} \left\vert \sum_{k=1}^{N}K_{k,j-1} f_{k}^{N}(t)-\sum_{k=1}^{\infty}
			K_{k, j-1} f_{k}(t)\right\vert &=0,\label{KC4}\\
				\lim_{N\rightarrow\infty} 	\sup_{t\in [0,T]} \left\vert \sum_{k=0}^{N-1}K_{1,k} f_{k}^{N}(t)-\sum_{k=0}^{\infty}
			K_{1,k} f_{k}(t)\right\vert &=0,\label{KC5}\\
			\text{ and}\quad	\lim_{N\rightarrow\infty} 	\sup_{t\in [0,T]} \left\vert \sum_{k=1}^{N}K_{k,0} f_{k}^{N}(t)-\sum_{k=1}^{\infty}
			K_{k,0} f_{k}(t)\right\vert &=0\label{KC6}.			
		\end{align}
	\end{linenomath}
		Upon expressing the truncated system \eqref{Tode0}--\eqref{Todej} in the integral form, the equations can be written as
		\begin{linenomath}
		\begin{align}
				f_{0}^{N}(t)&= f_{0}^{N}(0) + \int_{0}^{t} f_{1}^{N}(s)\sum_{k=0}^{N-1}K_{1,k} f_{k}^{N}(s) ds- \int_{0}^{t} f_{0}^{N}(s)\sum_{k=1}^{N}K_{k,0} f_{k}^{N}(s)ds,\nonumber 
	  \end{align}
  \end{linenomath}
   and
   \begin{linenomath}
	\begin{align}	
	f_{j}^{N}(t)  &  = f_{j}^{N}(0)+\int_{0}^{t} f_{j+1}^{N}(s)\sum
			_{k=0}^{N-1}K_{j+1,k} f_{k}^{N}(s) ds-\int_{0}^{t} f_{j}^{N}\sum_{k=0}%
			^{N-1}K_{j,k} f_{k}^{N}(s)ds \nonumber\\
			&\quad  -\int_{0}^{t} f_{j}^{N}(s)\sum_{k=1}^{N}K_{k,j} f_{k}^{N}(s) ds+\int_{0}^{t} f_{j-1}^{N}(s)\sum_{k=1}^{N}K_{k,j-1} f_{k}^{N}(s)ds,\nonumber
		\end{align}
	\end{linenomath}
	for $j=1,2,3, \ldots, N-1,N$.
		Now, we can take the limit as $N\rightarrow\infty$ in the  above equations, by using \eqref{KC1}--\eqref{KC6} and the uniform convergence of $(f_{j}^{N})_{j\ge 0}$ to $( f_{j})_{j\ge 0}$ on $[0,T]$. This demonstrates that  $( f_{j})_{j\ge 0}$, as the limit, serves as a mild solution to the EDG system (\ref{infode})-(\ref{infIC}). This completes the the proof of \Cref{mainthm-1}(a). Similarly, to prove \Cref{mainthm-1}(b), we only have to show that
		\begin{linenomath} 
		\begin{align}\label{coagtail}
			\lim_{M\rightarrow\infty} 	\sup_{t\in [0,T]} \left\vert \sum_{k=M+1}^{\infty}K_{j,k}( f_{k}(t)+ f_{k}^{N}(t))\right\vert = 0,
		\end{align}
	\end{linenomath}
		when $\max\{\mu,\nu\} \leq 1 \text{ and } \min\{\mu,\nu\} < 1$ with $f_{j}^{N}:=0$ for $j>N$. By using \eqref{quadraticcoagassum}, \eqref{mainconvexinequality-lamda3} and \eqref{solmainconvexinequality-lambda3}, we obtain
		\begin{linenomath}
		\begin{align}
			\left\vert \sum_{k=M+1}^{\infty}K_{j,k}( f_{k}(t)+ f_{k}^{N}(t))\right\vert
			&\leq 2C_{q}\sum_{k=M+1}^{\infty} jk ( f_{k}(t)+ f_{k}^{N}(t)) \nonumber\\
			&\leq 2C_{q}j\sum_{k=M+1}^{\infty}k( f_{k}(t)+ f_{k}^{N}(t)) \nonumber\\
			&\leq 2C_{q}j\sup_{k\geq M+1}\left(\frac{k}{G_{1}(k)}\right)\sum_{k=M+1}^{\infty}G_{1}(k)( f_{k}(t)+f_{k}^{N}(t))\nonumber\\
			&\leq 4C_{q}C_{3}(T)j\sup_{k\geq M+1}\left(\frac{k}{G_{1}(k)}\right).\nonumber
		\end{align}
	\end{linenomath} 
	This implies that
	\begin{linenomath} 
		\begin{align}
			\sup_{t\in [0,T]} \left\vert \sum_{k=M+1}^{\infty}K_{j,k}( f_{k}(t)+ f_{k}^{N}(t))\right\vert
		&\leq 4C_{q}C_{3}j\sup_{k\geq M+1}\left(\frac{k}{G_{1}(k)}\right).\nonumber
	\end{align}
\end{linenomath}
		Next, by passing the limit as $M\rightarrow \infty$ and using \eqref{vanishingconvexlimit}, we obtain \eqref{coagtail}, which completes the proof of \Cref{mainthm-1}(b).
	\end{proof}
	\begin{proof}[\textbf{Proof of \Cref{regularityandconservation}.}]
		Let $ (f_{j})_{j\ge 0}$ be the mild solution to \eqref{infode}--\eqref{infIC} under the
		conditions of \Cref{mainthm-1}. Then, from the construction in \Cref{mainthm-1}, consider the integral form of \eqref{infode}--\eqref{infIC}
		\begin{linenomath}
		\begin{align} 
			f_{0}(t)&= f_{0}(0)+\int_{0}^{t} f_{1}(s)\sum_{k=0}^{\infty}
			K_{1,k} f_{k}(s)ds - \int_{0}^{t} f_{0}(s)\sum_{k=1}^{\infty}K_{k,0} f_{k}(s) ds\nonumber\\
			\text{and}\quad	f_{j}(t)&= f_{j}(0)+\int_{0}^{t}   f_{j+1}(s)\sum_{k=0}^{\infty
			}K_{j+1,k} f_{k}(s)ds - \int_{0}^{t}f_{j}(s)\sum_{k=0}^{\infty}K_{j,k} f_{k}(s)  ds \nonumber \\
			&\quad -\int_{0}^{t}
			f_{j}(s)\sum_{k=1}^{\infty}K_{k,j} f_{k}(s) ds+\int_{0}^{t} f_{j-1}(s)\sum_{k=1}^{\infty}K_{k,j-1} f_{k}(s)ds, \nonumber
		\end{align}
	\end{linenomath}
	for all $j\in \mathbb{N}$. It is clear that the integrands on the right-hand sides of the above equations are continuous functions due to the uniform convergence of  $(f_{j}^{N})_{j \ge 0} $ to  $(f_{j})_{j \ge 0} $ on $[0,T]$ and the conditions in \eqref{KC1}--\eqref{KC6}.
	Therefore, the solution $(f_{j})_{j\ge 0}$ is differentiable by the fundamental theorem of calculus. Then, the conservation of number of particles, i.e., $M_{0}(f(t))=M_{0}(f(0))$ and total mass, i.e., $M_{1}(f(t))=M_{1}(f(0))=\rho$ of the system for all $t\in [0,T]$, is a consequence of \eqref{mainconvexinequality-lamda12}, \eqref{localparticleconservation}, \eqref{localmassconservation}, \eqref{solparticlemassbound}, \eqref{vanishingconvexlimit} and the uniform convergence of  $(f_{j}^{N})_{j \ge 0} $ to  $(f_{j})_{j \ge 0} $. In order to prove these conservation results, let us consider two cases. In the first case, assume that $f(0)=(f_{j}(0))_{j\ge 0} \in Y_{\lambda}^{+}$ with $\lambda = \max\{\mu,\nu\}>1$. First, we proceed to prove $M_{1}(f(t))=M_{1}(f(0))$ for all $t\in [0,T]$. For $N\geq M \ge 2$, from \eqref{localmassconservation}, \eqref{mainconvexinequality-lamda12}, and noting that $\lambda >1$, we have
	\begin{linenomath}
		\begin{align*}
			\left|\sum_{j=0}^{\infty}j\left(f_{j}(0)-f_{j}(t)\right)\right|  &\leq \left| \sum_{j=0}^{M-1}j\left(f_{j}^{N}(t)-f_{j}(t)\right)\right|+\sum_{j=M}^{N}jf_{j}^{N}(t)+\sum_{j=N+1}^{\infty}jf_{j}(0) +\sum_{j=M}^{\infty}jf_{j}(t)\\
			&\leq \left| \sum_{j=0}^{M-1}j\left(f_{j}^{N}(t)-f_{j}(t)\right)\right|+\sup_{j\geq M}\left(\frac{j}{G_{\lambda}(j)}\right)\sum_{j=M}^{N}j^{\lambda-1}G_{\lambda}(j)f_{j}^{N}(t)\\
			&\quad +\sum_{j=N+1}^{\infty}jf_{j}(0)+\sum_{j=M}^{\infty}jf_{j}(t)\\
			&\leq \left| \sum_{j=0}^{M-1}j\left(f_{j}^{N}(t)-f_{j}(t)\right)\right|+C_{1,2}(T)\sup_{j\geq M}\left(\frac{j}{G_{\lambda}(j)}\right)\\
			&\quad +\sum_{j=N+1}^{\infty}jf_{j}(0)+\sum_{j=M}^{\infty}jf_{j}(t).
		\end{align*}
	\end{linenomath}
Then, passing the limit as $N\rightarrow \infty$ and using the uniform convergence of $(f_{j}^{N})_{j\ge 0}$ to $(f_{j})_{j\ge 0}$ on $[0,T]$ and $M_{1}(f(0))<\infty$, we infer that
\begin{linenomath}
		\begin{align*}
			\left|\sum_{j=0}^{\infty}j\left(f_{j}(0)-f_{j}(t)\right)\right|  &\leq C_{1,2}(T)\sup_{j\geq M}\left(\frac{j}{G_{\lambda}(j)}\right) +\sum_{j=M}^{\infty}jf_{j}(t).
		\end{align*}
	\end{linenomath}
		Taking the limit as $M\rightarrow \infty$ and using \eqref{vanishingconvexlimit} and $M_{1}(t)<\infty$, we conclude that
		\begin{linenomath}
		\begin{align*}
			\sum_{j=0}^{\infty}j f_{j}(t)  =\sum_{j=0}^{\infty}j f_{j}(0).
		\end{align*}
	\end{linenomath}
		Similarly, for $N\geq M \ge 2$, using \eqref{localparticleconservation}, \eqref{localmassconservation} and \eqref{solparticlemassbound}, we set
		\begin{linenomath}
		\begin{align*}
			\left|\sum_{j=0}^{\infty}\left(f_{j}(0)-f_{j}(t)\right)\right|  &\leq \left| \sum_{j=0}^{M-1}\left(f_{j}^{N}(t)-f_{j}(t)\right)\right|+\sum_{j=M}^{N}f_{j}^{N}(t)+\sum_{j=N+1}^{\infty}f_{j}(0) +\sum_{j=M}^{\infty}f_{j}(t)\\
			&\leq \left| \sum_{j=0}^{M-1}\left(f_{j}^{N}(t)-f_{j}(t)\right)\right|+\frac{2}{M}M_{1}(f(0))+\sum_{j=N+1}^{\infty}jf_{j}(0).
		\end{align*}
	\end{linenomath}
		Again, taking the limit as $N \rightarrow \infty$ and using the uniform convergence of $(f_{j}^{N})_{j \ge 0}$ to $(f_{j})_{j \ge 0}$ on $[0,T]$ along with $M_{1}(f(0)) < \infty$, we infer that
		\begin{linenomath}
		\begin{align*}
			\left|\sum_{j=0}^{\infty}\left(f_{j}(0)-f_{j}(t)\right)\right|  &\leq \frac{2}{M}M_{1}(f(0)). 
		\end{align*}
	\end{linenomath}
		Finally, passing the limit as $M\rightarrow \infty$, we obtain
		\begin{linenomath}
		\begin{align*}
			\sum_{j=0}^{\infty}f_{j}(t)  =\sum_{j=0}^{\infty} f_{j}(0).
		\end{align*}
	\end{linenomath}
		Similarly, in the second case, i.e., when $ \max\{\mu,\nu\} \leq 1 \text{ and } \min\{\mu,\nu\} < 1$, the proofs for the conservation of the number of particles and the total mass of the system follow from \eqref{localparticleconservation}, \eqref{localmassconservation}, \eqref{solparticlemassbound}, \eqref{vanishingconvexlimit}, \eqref{mainconvexinequality-lamda3} and the uniform convergence of  $(f_{j}^{N})_{j \ge 0} $ to  $(f_{j})_{j \ge 0} $ on $[0,T]$. This completes the proof of \Cref{regularityandconservation}.	
	\end{proof}	
	\begin{proof} [\textbf{Proof of \Cref{localsolution}}]
		The proof of \Cref{localsolution} is based on the method analogous to that used in the proof of \Cref{mainthm-1}. Assuming that $K_{j,k}=K_{k,j} \leq C j^2 k^2$, let us first consider
		\begin{linenomath}
		\begin{equation}
			\frac{d}{dt}\left(\sum_{j=0}^{N} j^{2}f_{j}^{N}(t)\right) \leq 2C \sum_{j=0}^{N-1} \sum_{k=0}^{N-1} j^2 k^2 f_{j}^{N}(t) f_{k}^{N}(t) \leq 2C \left( \sum_{j=0}^{N} j^{2}f_{j}^{N}(t) \right)^2,\nonumber
		\end{equation}
	\end{linenomath}
		which implies that
		\begin{linenomath}
		\begin{align}\label{M2bounnd}
			\sum_{j=0}^{N} j^{2}f_{j}^{N}(t)\leq\frac{1}{\frac{1}{M_{2}^{N}(f(0))}-2Ct}\leq\frac{1}{\frac
				{1}{M_{2}(f(0))}-2Ct}\text{ for }t< \frac{1}{2M_{2}(f(0))C}.
		\end{align}
	\end{linenomath}
		Despite this, we can still construct a subsequence $(f_{j}^{N})_{j\ge 0}$ (not relabeled), as previously that converges uniformly to a limiting function  $(f_{j})_{j\ge 0}$ on $[0,T_{0}]$, where $0<T_{0}< \frac{1}{2M_{2}(f(0))C}$. Subsequently, we can establish an upper bound for $\sum_{j=0}^{\infty}jG_{2}(j)f_{j}^{N}(t)$ up to a finite time $T_{0}<\frac{1}{2M_{2}(f(0))C}$, for some $G_{2} \in \mathcal{G}_{1,\infty}$ satisfying $\sum_{j=0}^{\infty}jG_{2}(j)f_{j}(0)<\infty$, by using the de la Vall\'{e}e-Poussin theorem and the Gronwall inequality. Afterward, we can demonstrate that the partial sums in the truncated system  \eqref{Tode0}--\eqref{Todej} converge uniformly up to time $T_{0}$, which shows the existence of  continuously differentiable solution $( f_{j})_{j\ge 0}\in  Y_{2}^{+}$ to the infinite EDG system \eqref{infode}--\eqref{infIC} on $[0,T_{0}]$. Moreover, the truncated solution $(f_{j}^{N})_{j\ge 0}$ to \eqref{Tode0}--\eqref{TodeIC} satisfies \eqref{localparticleconservation} and \eqref{localmassconservation}. Hence, similarly to the proof of \Cref{regularityandconservation}, we can prove that  $ M_{0}(f(t)) = M_{0}(f(0))  $ and  $ M_{1}(f(t)) = M_{1}(f(0))  $ for all  $ t \in [0,T_{0}]  $, where  $ T_{0} < \frac{1}{2M_{2}(f(0))C}  $. This is accomplished by employing \eqref{localparticleconservation}, \eqref{localmassconservation}, \eqref{M2bounnd}, and  $ (f_{j})_{j \ge 0} \in Y_{2}^{+}  $, thereby completing the proof of \Cref{localsolution}.
	\end{proof}
	\section{Finite Time Gelation}
	\noindent Before proceeding to show the occurrence of finite time gelation, we need some higher moment bounds. Therefore, we prove the following propagation of moments result.
	\begin{lemma}[Propagation of moments]\label{momentpropagation}
		Consider the infinite EDG system \eqref{infode}--\eqref{infIC}. Let
		$K_{j,k}$ be the symmetric interaction kernel and satisfy $C_{1}\left(j^{2}k^{\alpha}+j^{\alpha}k^{2}\right)\leq K_{j,k}\leq Cj^{2}k^{2}$ for all $(j,k)\in N_{0}\times N_{0}$, where $1<\alpha\leq 2$ and $C_{1}>0$. If $f(0)=(f_{j}(0))_{j\ge 0} \in Y_{r}^{+}$ with $M_{r}(f(0))>0$ for $r\geq 2$, then the  continuously differentiable local solution to \eqref{infode}--\eqref{infIC} on $[0, T_{0}]$, constructed in \Cref{localsolution}, satisfies $( f_{j}(t))_{j\ge 0}\in  Y_{r}^{+}$ for all $t \in [0,T_{0}]$, where $0<T_{0}< \frac{1}{2M_{2}(f(0))C}$.
	\end{lemma}
	\begin{proof}
		From \Cref{localsolution}, it is evident that the system \eqref{infode}--\eqref{infIC} has a  continuously differentiable local solution $( f_{j})_{j\ge 0}\in  Y_{2}^{+}$ satisfying
		\begin{linenomath} 
		\begin{align}\label{propsecondM2}
			M_{2}(f_{j}(t)) \leq \frac{1}{\frac
				{1}{M_{2}(f(0))}-2CT_{0}},\quad \; \text{for each}\; t \in [0,T_{0}].
		\end{align}
	\end{linenomath}
		Hence, we only need to show that  $( f_{j}(t))_{j\ge 0}\in  Y_{r}^{+}$ for all $t \in [0,T_{0}]$ with $r \ge 2$. For $N>1$, define $h_{j}:= \min \left\{j^{r}, N^{r}\right\}$ for  $j\in \mathbb{N}_{0}$.
		Since $K_{j,k}$ is symmetric with $K_{j,0}=0=K_{k,0}$ for all $j, k \in \mathbb{N}$, we infer from \Cref{notionofsolution} that
		\begin{linenomath}
		\begin{align}\label{proptemp1}
			\sum_{j=0}^{\infty}h_{j} f_{j}(t)=&\sum_{j=0}^{\infty}h_{j} f_{j}(0)+\int_{0}^{t}\sum_{j=1}^{\infty} \left(h_{j+1}-2h_{j}+ h_{j-1}\right)\sum_{k=1}^{\infty}  K_{j,k} f_{j}(s) f_{k}(s)ds, 
		\end{align}
	\end{linenomath}
		for each $t\in [0, T_{0}]$. Then, using the non-negativity of $K_{j,k}$ and $(f_{j})_{j\ge 0}$, from \eqref{proptemp1}, we derive the following estimate
		\begin{linenomath}
			\begin{align}
			\sum_{j=1}^{N}j^{r} f_{j}(t)\leq &\sum_{j=1}^{\infty}j^{r} f_{j}(0)+\int_{0}^{t}\sum_{j=1}^{N-1} \left(h_{j+1}-2h_{j}+ h_{j-1}\right)\sum_{k=1}^{\infty}  K_{j,k} f_{j}(s) f_{k}(s)ds. \nonumber 
		\end{align}
	\end{linenomath}
		Then, by the mean value theorem \eqref{mean value theorem}, we get the identity
		\begin{linenomath}
		\begin{align}\label{meanvalue2}
		h_{j+1}-2h_{j}+ h_{j-1}=h^{\prime \prime }(\theta(j))=r(r-1)\theta(j)^{r-2}
		\end{align}
	\end{linenomath} 
		 for some $\theta(j) \in (j-1,j+1)$ with $j=1, 2,\ldots, N-1$.
		Therefore, by inserting \eqref{propsecondM2} and \eqref{meanvalue2} into \eqref{proptemp1} and noting that $r\ge 2$, we deduce the following estimate
		\begin{linenomath}
		\begin{align}
			\sum_{j=1}^{N}j^{r} f_{j}(t)\leq &\sum_{j=1}^{\infty}j^{r} f_{j}(0)+ Cr(r-1)\sum_{j=1}^{N-1}\sum_{k=1}^{\infty}  \theta(j)^{r-2}j^{2}k^{2} f_{j}(t) f_{k}(t)\nonumber\\
			&\leq M_{r}(f(0)) + Cr(r-1)\sum_{j=1}^{N-1}\sum_{k=1}^{\infty}  (j+1)^{r-2}j^{2}k^{2} f_{j}(t) f_{k}(t)\nonumber\\
			&\leq M_{r}(f(0)) + Cr(r-1)2^{r-2}\sum_{j=1}^{N}\sum_{k=1}^{\infty}  j^{r}k^{2} f_{j}(t) f_{k}(t) \nonumber\\
			& = M_{r}(f(0)) + Cr(r-1)2^{r-2} M_{2}(f(t))\sum_{j=1}^{N}j^{r}f_{j}(t) \leq \tilde{C_{\star}} \sum_{j=1}^{N}j^{r}f_{j}(t),\nonumber
		\end{align}
	\end{linenomath}
		where $\tilde{C_{\star}}:= \frac{Cr(r-1)2^{r-2} M_{2}(f(0))}{1-2CT_{0}M_{2}(f(0))}$. 
		Subsequently, by applying Gronwall's lemma, we derive
		\begin{linenomath}
		\begin{align}
			\sum_{j=1}^{N}j^{r}f_{j}(t)\leq M_{r}(f(0))e^{\tilde{C_{\star}}t}.\nonumber
		\end{align}
	\end{linenomath}
		By passing the limit as $N \rightarrow \infty$, we finally obtain
		\begin{linenomath} 
		\begin{align}
			M_{r}(f(t))\leq M_{r}(f(0))e^{\tilde{C_{\star}}t},\;\; t\in [0, T_{0}].\nonumber
		\end{align}
	\end{linenomath}
		This completes the proof of \Cref{momentpropagation}.
	\end{proof}
	With this preparation, we are now ready to provide the proof of \Cref{finitetimegel}.
	\begin{proof} [\textbf{Proof of \Cref{finitetimegel}}]
		Since $M_{2+\alpha}(f(0))<\infty$, from \Cref{momentpropagation}, we have a continuously differentiable local solution  $( f_{j}(t))_{j\ge 0}$ to \eqref{infode}--\eqref{infIC} in $Y_{2+\alpha}^{+}$, for $t< \frac{1}{2M_{2}(f(0))C}$, such that
		\begin{linenomath}
		\begin{align}\label{alplaequation}
			M_{\alpha}(f(t))&=	M_{\alpha}(f(0))+\int_{0}^{t}\sum_{j=1}^{\infty}( (j+1)^{\alpha}%
			-2j^{\alpha}+ (j-1)^{\alpha}) f_{j}(s)\sum_{k=1}^{\infty}K_{j,k} f_{k}(s)ds,
		\end{align}
	\end{linenomath}
		for $1<\alpha\leq2$. By using the 
		mean value theorem \eqref{mean value theorem} in the above equation, we infer that
		\begin{linenomath} 
		\begin{align}
			M_{\alpha}(f(t))=M_{\alpha}(f(0))+\alpha (\alpha-1)\int_{0}^{t}\sum_{j=1}^{\infty}\sum_{k=1}^{\infty}\theta(j) ^{\alpha-2}K_{j,k} f_{j}(s) f_{k}(s)ds,\nonumber
		\end{align}
	\end{linenomath}
		where $\theta(j) \in (j-1,\;j+1)$.
		Since $\alpha-2\leq 0$ and $\alpha>1$, we have 
		\begin{linenomath}
		\begin{align}
			M_{\alpha}(f(t))&\geq M_{\alpha}(f(0))+C_{1}\alpha (\alpha-1)\int_{0}^{t}\sum_{j=1}^{\infty}\sum_{k=1}^{\infty}(j+1) ^{\alpha-2}j^{2}k^{\alpha} f_{j}(s) f_{k}(s)ds \nonumber\\ 
			&\geq  M_{\alpha}(f(0))+C_{1}\alpha (\alpha-1)2^{\alpha-2} \int_{0}^{t}M_{\alpha}^{2}(f(s))ds.\nonumber
		\end{align}
	\end{linenomath}
		Solving the above integral inequality, we obtain
		\begin{linenomath}
		\begin{align}
			M_{\alpha}(f(t))\geq \left[\frac{1}{M_{\alpha}(f(0))}-C_{1}\alpha (\alpha-1)2^{\alpha-2}t\right]^{-1}.\nonumber
		\end{align}
	\end{linenomath}
This inequality shows that  $ M_{\alpha}(f(t))  $ blows up at  $ t = \frac{1}{C_{1} \alpha (\alpha - 1) 2^{\alpha - 2} M_{\alpha}(f(0))}  $. Therefore,  $ T_{gel} \leq \frac{1}{C_{1} \alpha (\alpha - 1) 2^{\alpha - 2} M_{\alpha}(f(0))}  $, indicating the occurrence of finite-time gelation. Next, to prove that  $ T_{gel} = \frac{1}{2M_{2}(f(0))C}  $ when  $ K_{j,k} = C j^{2} k^{2}  $ for all  $ (j,k) \in \mathbb{N}_{0} \times \mathbb{N}_{0}  $, we set  $ \alpha = 2  $ and  $ C_{1} = \frac{C}{2}  $. From \eqref{alplaequation}, we then deduce the following equation
\begin{linenomath}
\begin{align}
	M_{2}(f(t)) &= M_{2}(f(0)) + 2C \int_{0}^{t} \sum_{j=1}^{\infty} \sum_{k=1}^{\infty} j^{2} k^{2} f_{j}(f(s)) f_{k}(s) \, ds \nonumber \\
	&= M_{2}(f(0)) + 2C \int_{0}^{t} M_{2}^{2}(s) \, ds, \nonumber
\end{align}
\end{linenomath}
for   $0\leq t < \frac{1}{2M_{2}(f(0))C} $. Solving the above integral equality, we find that
\begin{linenomath}
\begin{align}
	M_{2}(f(t)) = \left[\frac{1}{M_{2}(f(0))} - 2Ct\right]^{-1}, \nonumber
\end{align}
\end{linenomath}
which implies  $ T_{gel} \leq \frac{1}{2M_{2}(f(0))C}  $ as  $ M_{2}(f(t))  $ blows up at  $ t = \frac{1}{2M_{2}(f(0))C}  $. However, from \Cref{localsolution}, we have  $ T_{gel} \geq \frac{1}{2M_{2}(f(0))C}  $. Therefore, we conclude that  $ T_{gel} = \frac{1}{2M_{2}(f(0))C}  $. This completes the proof of \Cref{finitetimegel}.
\end{proof}
\begin{proof}[\textbf{Proof of \Cref{nonexistenceresult-1}}]
		Suppose that the infinite EDG system \eqref{infode}--\eqref{infIC} has a global mass conserving solution $(f_{j}(t))_{j\ge 0}$ in the space $Y_{2}^{+}$, i.e., $M_{2}(f(t))<\infty$ for all $t\in (0,\infty)$. This is a contradiction to the blow up of $M_{\alpha}(f(t))$ at $t=\frac{1}{C_{1}\alpha (\alpha-1)2^{\alpha-2} M_{\alpha}(f(0))}$, for $1< \alpha \le 2$.
	\end{proof}
	\section{Instantaneous Gelation}
	\noindent Finally, we focus on the possible occurrence of instantaneous gelation, which leads to the non-existence of solutions to \eqref{infode}--\eqref{infIC}. The work done in \cite{Naim} indicates that specific interaction kernels of the form $K_{j,k}=j^{\mu}k^{\nu}+j^{\nu}k^{\mu}$ with $\mu,\nu>2$ lead to the occurrence of instantaneous gelation phenomena. Building on approaches from \cite{Esenturk, Ball}, we demonstrate that instantaneous gelation occurs with faster-growing interaction kernels. Understanding the tail behavior of the solution to  \eqref{infode}--\eqref{infIC} with such faster-growing  interaction kernels is essential for proving \Cref{Tgel}. In order to control the tail of the solution, we consider the following infinite system
	\begin{linenomath}
	\begin{align*}
		\dot{ f}_{j}=I_{j-1}( f)-I_{j}( f)\quad \text{for} \; j\in \mathbb{N},
	\end{align*}
\end{linenomath}
	where
	\begin{linenomath}
	\begin{equation}
		I_{j}( f)= f_{j}\sum_{k=1}^{\infty}K_{k,j} f_{k}- f_{j+1}\sum_{k=0}^{\infty
		}K_{j+1,k} f_{k}. \label{flow-equ}
	\end{equation}
\end{linenomath}
	The following lemmas are crucial for establishing the aforementioned \Cref{Tgel}. We now present the result from \cite[Lemma 4]{Esenturk}.
	\begin{lemma}
		\label{L-tail}Let $ (f_{j}(t))_{j\ge 0} \in Y_{2}^{+}$ be a mild solution of the EDG system \eqref{infode}--\eqref{infIC} for $t\in [0,T]$ and $0<T< T_{gel}$. Then, we have the following identities for $0<\sigma <  t\leq T$ and $m\ge 2$
		\begin{linenomath}
		\begin{align*}
			\sum_{j=m}^{\infty} f_{j}(t)-\sum_{j=m}^{\infty} f_{j}(\sigma)  &  =\int_{ \sigma}%
			^{t}I_{m-1}( f(s))ds,\\
			\sum_{j=m}^{\infty}j f_{j}(t)-\sum_{j=m}^{\infty}j f_{j}( \sigma)  &  =\int_{ \sigma}%
			^{t}\sum_{j=m}^{\infty}I_{j}( f(s))ds+m\int_{ \sigma}^{t}I_{m-1}( f(s))ds,\\
		\sum_{j=m}^{\infty}j^{2} f_{j}(t)-\sum_{j=m}^{\infty}j^{2} f_{j}( \sigma)  &
			=\int_{ \sigma}^{t}\sum_{j=m}^{\infty}(2j+1)I_{j}( f(s))ds+m^{2}\int_{ \sigma}^{t}%
			I_{m-1}( f(s))ds.
		\end{align*}
		\end{linenomath}
	\end{lemma}
	\begin{lemma}\label{monotoneM2}
		Let  $(f_{j}(t))_{j\ge 0} \in Y_{2}^{+}$ be a mild solution of the EDG system \eqref{infode}--\eqref{infIC} for $t\in [0,T]$ and $0<T< T_{gel}$. Assume that $K_{j,k}$ is a non-negative symmetric interaction kernel satisfying $K_{j,0}=0$ for all $j \in \mathbb{N}_{0}$. Then  $ M_{0}(f(t)) = M_{0}(f(0))  $ and  $ M_{1}(f(t)) = M_{1}(f(0))  $ for any  $ t \in [0, T]  $, and  $ M_{2}(f(\cdot))  $ is a non-decreasing function on  $ [0, T]  $.
	\end{lemma}
	\begin{proof}
	Consider $0<T< T_{gel}$ and  $0\leq t_{1}\leq t_{2} \leq T$. Since $K_{j,k}$ is symmetric with $K_{j,0}=0=K_{k,0}$ for all $j, k \in \mathbb{N}$, we infer from \Cref{notionofsolution} that
	\begin{linenomath}
		\begin{align}\label{infinitediveform}
			\sum_{j=0}^{\infty}h_{j} f_{j}(t_{2})=&\sum_{j=0}^{\infty}h_{j} f_{j}(t_{1})+\int_{t_{1}}^{t_{2}}\sum_{j=1}^{\infty} \left(h_{j+1}-2h_{j}+ h_{j-1}\right)\sum_{k=1}^{\infty}  K_{j,k} f_{j}(s) f_{k}(s)ds, 
		\end{align}
	\end{linenomath}
	 for any non-negative sequence $(h_{j})_{j\ge 0}$. Therefore, $M_{0}(f(t))=M_{0}(f(0))$ for any $t\in [0,T]$ can be proved by inserting and $t_{1}=0$, $t_{2}=t$ and $h_{j}:=1$, for all $j\in \mathbb{N}_{0}$ into \eqref{infinitediveform}. Similarly, $M_{1}(f(t))=M_{1}(f(0))$  can be proved by taking $h_{j}:=j$, for all $j\in \mathbb{N}_{0}$ in \eqref{infinitediveform}. Next, by applying $h_{j}:=j^{2}$ for all $j \in \mathbb{N}_{0}$ and using the non-negativity of $(f_{j})_{j\ge 0}$ and $K_{j,k}$, we obtain
	 \begin{linenomath}
	 	\begin{align}
	 			\sum_{j=1}^{\infty}j^{2} f_{j}(t_{2})-\sum_{j=1}^{\infty} j^{2} f_{j}(t_{1})= 2\int_{t_{1}}^{t_{2}}\sum_{j=1}^{\infty}f_{j}(s)\sum_{k=1}^{\infty}K_{j,k}f_{k}(s)ds\geq 0.\nonumber
	 	\end{align}
 	\end{linenomath}
 	It implies that
 	\begin{linenomath}
 		\begin{align}
 		\sum_{j=1}^{\infty}j^{2} f_{j}(t_{2}) \ge \sum_{j=1}^{\infty} j^{2} f_{j}(t_{1}).\nonumber
 	\end{align}
 \end{linenomath} 
Hence, $M_{2}(f(\cdot))$ is a non-decreasing function of time.
	\end{proof}
	In order to show the occurrence of instantaneous gelation, the upcoming lemma shows that all higher-order moments are finite.
	\begin{lemma}\label{finintemomentproperty}
	Let $K_{j,k}$ be a non-negative symmetric interaction kernel satisfying $K_{j,0}=0$ for all $j \in \mathbb{N}_{0}$, with $K_{j,k}\geq C\left(j^{\beta}+k^{\beta}\right)$ for some $C>0$ and $\beta>2$. Suppose  $ (f_{j}(t))_{j \ge 0}\in  Y_{2}^{+}$ is a mild solution to \eqref{infode}--\eqref{infIC} on $[0,T]$ for some $0<T< T_{gel}$, with the initial condition $ f(0) \in  Y_{2}^{+}$, $\sum_{k=0}^{\infty} f_{k}(0)\geq C_{2}>0$. Then, for any $p\in \mathbb{N}$, $M_{p}(f(t))<\infty$ for all $t\in [0,T)$. 
	\end{lemma}
	\begin{proof}
		Let $(f_{j}(t))_{j\ge 0}\in  Y_{2}^{+}$ be a solution to \eqref{infode}--\eqref{infIC} on $[0,T]$. Then, $M_{2}(f(t))<\infty$ for $t<T$. By utilizing the first and third identities of Lemma \ref{L-tail}, we obtain
		\begin{linenomath}
		\begin{equation}
			\sum_{j=m}^{\infty}(j^{2}-m^{2}) f_{j}(t)-\sum_{j=m}^{\infty}(j^{2}-m^{2}%
			) f_{j}( \sigma)=\int_{ \sigma}^{t}\sum_{j=m}^{\infty}(2j+1)I_{j}( f(s))ds,\label{tail-df2}%
		\end{equation}
	\end{linenomath}
		for $0<\sigma<t \leq T$. By extracting $I_{j}( f(s))$ from \eqref{flow-equ} and moving it to the right-hand side of \eqref{tail-df2}, and then adjusting the index for the term $ f_{j+1}$, the expression becomes
		\begin{linenomath}
		\begin{align*}
			\int_{ \sigma}^{t}\sum_{j=m}^{\infty}(2j+1)I_{j}( f(s))ds &  =\int_{ \sigma}^{t}\sum
			_{j=m}^{\infty}(2j+1) f_{j}(s)\sum_{k=1}^{\infty}K_{k,j} f_{k}%
			(s)ds\label{flow-sh2}\\
			&  -\int_{ \sigma}^{t}\sum_{j=m+1}^{\infty}(2j-1) f_{j}(s)\sum_{k=0}^{\infty
			}K_{j,k} f_{k}(s)ds.
		\end{align*}
	\end{linenomath}
		Using the symmetry of the interaction kernel $K_{j,k}$, the conditions $K_{0,j}=0=K_{j,0}$ and non-negativity of $f_{j}$ and $K_{j,k}$, we obtain
		\begin{linenomath}
		\begin{align*}
			\int_{ \sigma}^{t}\sum_{j=m}^{\infty}(2j+1)I_{j}( f(s))ds & =
			2\int_{ \sigma}^{t}\sum_{j=m}^{\infty} f_{j}(s)\sum_{k=0}^{\infty}K_{k,j} f_{k}(s) + \int_{ \sigma}^{t}(2m-1) f_{m}(s)\sum_{k=1}^{\infty}K_{k,m}f_{k}(s)ds\\
			&\quad +
			\int_{ \sigma}^{t}\sum_{j=m}^{\infty}(2j-1) f_{j}(s)\sum_{k=1}^{\infty
			}(K_{k,j}-K_{j,k}) f_{k}(s)ds\\
			&\geq 	2\int_{ \sigma}^{t}\sum_{j=m}^{\infty} f_{j}(s)\sum_{k=0}^{\infty}K_{k,j} f_{k}(s).
		\end{align*}
	\end{linenomath}
		Then,  inserting  the above inequality into \eqref{tail-df2}, we observe
		\begin{linenomath}
		\begin{equation}
			\sum_{j=m}^{\infty}(j^{2}-m^{2}) f_{j}(t)-\sum_{j=m}^{\infty}(j^{2}-m^{2}%
			) f_{j}( \sigma)\geq2\int_{ \sigma}^{t}\sum_{j=m}^{\infty} f_{j}(s)\sum_{k=0}^{\infty
			}K_{k,j} f_{k}(s)ds.\label{mom-ters}%
		\end{equation}
	\end{linenomath}
		From \eqref{mom-ters} and \Cref{monotoneM2}, we can deduce the following by using the fact that $K_{j,k}>Cj^{\beta}$
		\begin{linenomath}
		\begin{align*}
			\sum_{j=m}^{\infty}j^{2} f_{j}(\sigma) &  \leq\sum_{j=m}^{\infty}j^{2} f_{j}(t)+\sum_{j=m}^{\infty}m^{2}
			f_{j}(\sigma)-2C\int_{ \sigma}^{t}\sum_{j=m}^{\infty}j^{\beta} f_{j}(s)\sum_{k=0}
			^{\infty} f_{k}(s)ds\\
			&  \leq 2M_{2}(f(T))+2CC_{2}m^{\beta-2}\int_{t}^{\sigma}
			\sum_{j=m}^{\infty}j^{2} f_{j}(s)ds.
		\end{align*}
	\end{linenomath}
		In the second line, we utilized $\sum_{k=0}^{\infty} f_{k}(s)=M_{0}(f(0))\geq C_{2}>0$ by \Cref{monotoneM2}. Solving the above integral inequality, we infer that
		\begin{linenomath}
		\begin{align*}
			\sum_{j=m}^{\infty}j^{2} f_{j}(\sigma)\leq 2M_{2}(f(T))e^{-2CC_{2}m^{\beta-2}(t-\sigma)}.
		\end{align*}
	\end{linenomath}
		Using $m^{2} f_{m}(t)\leq \sum_{j=m}^{\infty}j^{2} f_{j}(t)$ and the above inequality, we get
		\begin{linenomath}  
		\begin{align}\label{negative-exponentialseries}
		\sum_{j=m}^{\infty}j^{p} f_{j}(\sigma)\leq 2M_{2}(f(T))\sum_{j=m}^{\infty}j^{p-2}e^{-2CC_{2}j^{\beta-2}(t-\sigma)},
		\end{align}
	\end{linenomath}
		for $0<\sigma<t \leq T$.
		Since for $\beta>2$, we have
		\begin{linenomath} 
		\begin{align*}
			\lim_{j\rightarrow\infty}\frac{j^{\beta-2}}{\log j}=\infty. 
		\end{align*}
	\end{linenomath}
		Therefore, the series on the right-hand side of \eqref{negative-exponentialseries} is convergent, which implies that $M_{p}(f(t))<\infty$ for all $t\in [0,T)$ and for any $p\in \mathbb{N}$.
	\end{proof}
	At this juncture, we are ready to wrap up the proof of \Cref{Tgel}.
	\begin{proof}[\textbf{Proof of \Cref{Tgel}}]
		From the EDG system \eqref{infode}--\eqref{infIC}, we get
		\begin{linenomath}
		\begin{align}
			M_{n}(f(t))&=	M_{n}(f(0))+\int_{0}^{t}\sum_{j=1}^{\infty}((j+1)^{n}
			-2j^{n}+(j-1)^n) f_{j}(s)\sum_{k=1}^{\infty}K_{j,k} f_{k}(s)ds,\nonumber
		\end{align}
	\end{linenomath}
		for every natural number $n\ge 2$ and $t\in [0,T]$ with $0<T< T_{gel}$. All the terms in the above equation are finite by  Lemma \ref{finintemomentproperty}.
		Since
		\begin{linenomath}
		\begin{align*} 
			(j+1)^{n}-2j^{n}+(j-1)^n\geq n(n-1)j^{n-2}\quad \text{and}\quad K_{j,0}=0,\quad  \text{forall} \; j\in \mathbb{N}_{0},
		\end{align*}
	\end{linenomath}
		therefore, we obtain
		\begin{linenomath}
		\begin{align}
			M_{n}(f(t)) \geq M_{n}(f(0))+ n(n-1)\int_{0}^{t} \sum_{j=1}^{\infty} j^{n-2}f_{j}(s)\sum_{k=0}^{\infty}K_{j,k} f_{k}(s)ds. \nonumber
		\end{align}
	\end{linenomath}
	By using $K_{j,k}\ge C j^{\beta}$, \Cref{monotoneM2} and $M_{0}(f(0))\geq C_{2}>0$, the above inequality can be further estimated as
	\begin{linenomath} 
		\begin{align}
				M_{n}(f(t)) &\geq M_{n}(f(0))+ Cn(n-1)\int_{0}^{t} \sum_{j=1}^{\infty} j^{n-2+\beta}f_{j}(s)\sum_{k=0}^{\infty} f_{k}(s)ds. \nonumber\\
			&\geq M_{n}(f(0))+ CC_{2}n(n-1)\int_{0}^{t}M_{n-2+\beta}(f(s))ds.\label{Jensen-1}
		\end{align}
	\end{linenomath}
		An application of Jensen's inequality (see Appendix A) gives
		\begin{linenomath}
		\begin{align*}
			M_{n-2+\beta}(f(s)) \ge M_{1}^{-\Lambda}(f(0))\left(M_{n}(f(s))\right)^{1+\Lambda} \quad \text{for} \; s \in [0,t],
		\end{align*}
	\end{linenomath}
		with $\Lambda:=\frac{\beta -2}{n-1}$. Inserting the above inequality into \eqref{Jensen-1}, we obtain
		\begin{linenomath}
		\begin{align*}
			M_{n}(f(t)) \geq M_{n}(f(0))+ CC_{2}n(n-1)M_{1}^{-\Lambda}(f(0))\int_{0}^{t}\left(M_{n}(f(s))\right)^{1+\Lambda}ds.
		\end{align*}
	\end{linenomath}
		Then, by solving the above integral inequality, we get
		\begin{linenomath}
		\begin{align}
			M_{n}(f(t))&\geq \left[M_{n}^{-\Lambda}(f(0))-CC_{2}M_{1}^{-\Lambda}(f(0))(\beta-2)nt\right]^{\frac{-1}{\Lambda}}.\nonumber
		\end{align}
	\end{linenomath}
		The last inequality shows that $M_{n}(f(t))$  blows up at $t=\left( \frac{M_{n}(f(0))}{M_{1}(f(0))}\right)^{-\Lambda}\frac{1}{CC_{2}(\beta -2)n}$. 	Since $M_{n}(f(0))\geq M_{1}(f(0))$ for $n\ge 2$, the gelation time $T_{gel}\leq  \frac{C}{CC_{2}(\beta -2)n}$ for all $n\ge 2$. This implies that $T_{gel}=0$ by taking the limit as $n\rightarrow \infty$.
	\end{proof}
	An immediate consequence of the aforementioned results in \Cref{finintemomentproperty} and \Cref{Tgel} is \Cref{nonexistenceresult-2}.
	\begin{proof}[\textbf{Proof of \Cref{nonexistenceresult-2}}]
		For the sake of contradiction, suppose that $ (f_{j}(t))_{j\ge 0}\in  Y_{2}^{+}$ is a solution to \eqref{infode}--\eqref{infIC}
		on some interval $[0,T)$ for $T>0$. Clearly, \Cref{finintemomentproperty} implies that $M_{n}(f(t))<\infty$ on $[0,T)$ for any $n \in \mathbb{N}$. However, from \Cref{Tgel}, we have $T_{gel}=0$, which contradicts $M_{n}(f(t))<\infty$ on $[0,T)$ for any $n \in \mathbb{N}$. Therefore, 	there is no solution $ (f_{j}(t))_{j\ge 0}\in  Y_{2}^{+}$ of \eqref{infode}--\eqref{infIC}
		on any interval $[0,T)$. 
	\end{proof}	
\noindent	\textbf{Appendix A}.
	Let  $ q_{j} \geq 0  $ satisfy  $ \sum_{j=1}^{\infty} q_{j} = 1  $, and let  $ \Psi(x)  $ be a convex function. Then, by Jensen's inequality, we have
	\begin{linenomath}
	\begin{equation*}
	\sum_{j=1}^{\infty} q_{j} \Psi(x_{j}) \geq \Psi\left(\sum_{j=1}^{\infty} q_{j} x_{j}\right).
	\end{equation*}
\end{linenomath}
	Since  $ M_{1}(f(s)) = M_{1}(f(0))  $ for all  $ s \in [0, t] \subset [0, T)  $ by \Cref{monotoneM2}, we take  $ q_{j} = \frac{jf_{j}(s)}{M_{1}(f(0))}  $,  $ x_{j} = j^{n-1}  $, and  $ \Psi(x) = x^{1+\Lambda}  $ (where  $ \Lambda := \frac{\beta - 2}{n-1} > 0  $ for  $ \beta > 2  $) in Jensen's inequality. 
	It then follows that
	\begin{linenomath}
	\begin{equation*}
	M_{n-2+\beta}(f(s)) \geq M_{1}^{-\Lambda}(f(0)) \left(M_{n}(f(s))\right)^{1+\Lambda}, \quad \text{for} \; s \in [0, t].
	\end{equation*}
\end{linenomath}
	\noindent\textbf{Acknowledgements}. The research of SS was supported by Council of Scientific and Industrial Research, India, under the grant agreement No.09/143(0987)/2019-EMR-I.

	\bibliographystyle{abbrv}
	\bibliography{Exchnagedrivengrowth}
	
\end{document}